\documentclass[12pt]{amsart}
\usepackage[dvipsnames]{color}
\usepackage{amsfonts,amssymb,amsmath,amscd,amstext}
\usepackage[
citecolor=blue,
colorlinks=true,
final=true,                    
linkcolor=blue,
plainpages=false,              
pdfpagelabels=true,            
pdfencoding=auto,              
unicode=true,                  
hypertexnames=true,            
naturalnames=true              
]{hyperref}
\usepackage{amsmath}
\usepackage{cleveref}
\usepackage[utf8]{inputenc}
\usepackage{graphicx}
\usepackage{changes}
\usepackage{comment}
\usepackage{mathdots}
\usepackage[a4paper,top=2.5cm,bottom=2.5cm,left=2cm,right=2cm]{geometry}

\usepackage[normalem]{ulem} 
\usepackage[multiuser]{fixme}

\usepackage{xparse}
\let\realItem\item 
\makeatletter
\NewDocumentCommand\myItem{ o }{%
   \IfNoValueTF{#1}%
      {\realItem}
      {\realItem[#1]\def\@currentlabel{#1}}
}
\makeatother

\usepackage{enumitem}    
\setlist[enumerate]{
    before=\let\item\myItem,       
    label=\textnormal{(\arabic*)}, 
    widest=iii.                    
}
\FXRegisterAuthor{ap}{anap}{Andrea} 
\FXRegisterAuthor{nd}{annd}{Nico} 
\FXRegisterAuthor{sm}{ansm}{Sonia} 

\usepackage[all,cmtip]{xy}
\usetikzlibrary{matrix,calc}
\usepackage[autostyle, english = american]{csquotes}
\MakeOuterQuote{"} 

\newtheorem{theorem}{\bf Theorem}

\newtheorem{proposition}[theorem]{\bf Proposition}
\newtheorem{lemma}[theorem]{\bf Lemma}
\newtheorem{definition}[theorem]{\bf Definition}
\newtheorem{remark}[theorem]{\bf Remark}

\newcommand{\bR}{{\mathbb{R}}}
\newcommand{\bH}{{\mathbb{H}}}
\newcommand{\bE}{{\mathbb{E}}}

\newcommand{\bC}{{\mathbb{C}}}
\newcommand{\bN}{{\mathbb{N}}}

\newcommand{\bP}{{\mathbb{P}}}

\newcommand{\cF}{{\mathcal{F}}}

\newcommand{\Hi}{{\mathcal{H}}}

\newcommand{\dr}{\mathrm{d}}

\newcommand{\Div}{\operatorname{div}}
\newcommand{\rot}{\operatorname{rot}}

\title[Chernoff for the heat and the Schr\"odinger equation in the Heisenberg group]{Chernoff solutions of the heat and the Schr\"odinger equation in the Heisenberg group}
\author[Nicol\`o Drago]{Nicol\`o Drago}
	\address[Nicol\`o Drago]{Department of Mathematics, University of Genova,  Via Dodecaneso, 35, 16146 Genova GE, Italy}
	\email[Nicol\`o Drago]{nicolo.drago@unige.it}
\author[Sonia Mazzucchi]{Sonia Mazzucchi}
	\address[Sonia Mazzucchi]{Department of Mathematics, University of Trento, Via Sommarive 14, 38123 Povo (Trento), Italy}
	\email[Sonia Mazzucchi]{sonia.mazzucchi@unitn.it}
\author[Andrea Pinamonti]{Andrea Pinamonti}
	\address[Andrea Pinamonti]{Department of Mathematics, University of Trento, Via Sommarive 14, 38123 Povo (Trento), Italy}
	\email[Andrea Pinamonti]{andrea.pinamonti@unitn.it}

\begin{document}
\begin{abstract}
    This paper investigates the application of the classical Chernoff’s theorem to construct explicit solutions for the heat and Schrödinger equations on the Heisenberg group $\mathbb{H}^d$. Using semigroup approximation techniques, we obtain analytically tractable and numerically implementable representations of fundamental solutions. In particular, we establish a new connection between the heat equation and Brownian motion on $\mathbb{H}^d$ and provide a rigorous realization of the Feynman path integral for the Schrödinger equation. The study highlights the challenges posed by the noncommutative structure of the Heisenberg group and opens new directions for PDEs on sub-Riemannian manifolds.
\end{abstract}

\maketitle

\section{Introduction}
The study of partial differential equations (PDEs) on noncommutative spaces has attracted significant attention in recent years, as such problems naturally arise in various branches of mathematical physics, harmonic analysis, and geometric analysis \cite{Bar,BLU07}. Among these, the Heisenberg group $\mathbb{H}^d$ plays a fundamental role as a prototype of a noncommutative, sub-Riemannian manifold, providing a rich structure that generalizes Euclidean settings while capturing essential features of quantum mechanics and hypoelliptic analysis. In this context, the heat equation and the Schr\"odinger equation on the Heisenberg group are of particular interest, as they are closely related to fundamental operators in sub-Riemannian geometry, such as the sub-Laplacian and the Rockland operators (see \cite{BLU07,Cap, Fis} and references therein).

The purpose of this paper is to construct explicit solutions to the heat equation and the Schr\"odinger equation on the Heisenberg group using Chernoff's theorem, a powerful tool from semigroup theory. In particular, for what concern the heat equation we will study the Chernoff approximation in different Banach spaces, i.e. $L^2$ and $C_0$. 
The theorem provides a method for approximating strongly continuous semigroups, which are often used to describe the evolution of solutions to parabolic and dispersive PDEs. Using this approach, we obtain approximations of the fundamental solutions of the heat and Schr\"odinger equations that are both analytically tractable and numerically implementable.
Furthermore, the particular Chernoff-type approximations presented in this paper allow for the construction of related functional integral representations of the solutions. Indeed, in the case of the heat equation, we show how the approximations for the classical solution in the Banach space $C_0(\bH^d)$  lead to the construction of the Brownian motion on the Heisenberg group as the weak limit of a sequence of piecewise-geodesic random walks. Moreover, in the case of the Schr\"odinger equation on $L^2(\bH^d)$ we provide a rigorous mathematical realization of the heuristic Feynman path integral representation for the solution. This problem is highly non-trivial and, to the best of our knowledge, it is tackled here for the first time in a sub-Riemannian setting.

The application of Chernoff's theorem to the Heisenberg group presents several challenges and interesting features. Unlike the Euclidean setting, where the Laplacian generates a simple Gaussian heat kernel, the sub-Laplacian on the Heisenberg group exhibits anisotropic diffusion due to the underlying noncommutative structure. This requires careful adaptation of Chernoff's approximation techniques to the framework of the Heisenberg group. Furthermore, in the case of the Schrödinger equation, the presence of nontrivial commutation relations in the Heisenberg algebra leads to a distinct form of wave propagation, which differs significantly from the classical Schrödinger evolution in Euclidean space.

This work contributes to the existing literature by providing a constructive approach to heat and Schrödinger equations in the Heisenberg group, complementing previous studies that have relied on Fourier transform techniques, representation theory, and functional calculus; see, for example, \cite{ G1,G2,G3,G4,Fis}. Our results highlight the applicability of Chernoff's theorem in the setting of sub-Riemannian geometry and open new directions for further investigation, including potential extensions to other classes of nilpotent Lie groups and their associated PDEs.

\vspace{10pt}

The structure of the paper is as follows. In Section 2, we provide background information on some well-known notions in semigroup theory, set the notation for the Heisenberg group, and recall some useful facts. Section 3 is dedicated to constructing solutions for the heat equation using Chernoff's theorem. In Section 4, we also employ Chernoff's theorem to construct solutions for the Schrödinger equation, provide a Feynman-Kac formula, and interpret the results through Feynman path integrals. In Section 5, we rigorously realize the Feynman path integral representation as an infinite-dimensional oscillatory integral on a suitable Hilbert space of continuous paths. Finally, in Section 6, using the results from Section 2, we present a new construction of Brownian motion in the Heisenberg group.

\vspace{2pt}

\textbf{Acknowledgements:} 
S. Mazzucchi and A. Pinamonti are members of {\em Gruppo Nazionale per l'Analisi Matematica, la Probabilit\`a e le loro Applicazioni} (GNAMPA), of the {\em Istituto Nazionale di Alta Matematica} (INdAM), and A.P. was partially supported by MIUR-PRIN 2017 Project \emph{Gradient flows, Optimal Transport and Metric Measure Structures} and by the European Union under NextGenerationEU. PRIN 2022 Prot. n. 2022F4F2LH.
N. Drago is a member of {\em Gruppo Nazionale per la Fisica Matematica} (GNFM), of the {\em Istituto Nazionale di Alta Matematica} (INdAM).

\vspace{2pt}

\textbf{Data availability statement:}
Data sharing is not applicable to this article as no new data were created or analysed in this study.

\vspace{2pt}

\textbf{Conflict of interest statement:}
The authors certify that they have no affiliations with or involvement in any organization or entity with any financial interest or non-financial interest in
the subject matter discussed in this manuscript.

\section{Preliminary material}

\subsection{$C_0$ semigroups, evolution equations and Chernoff Theorem}
 In the following we shall denote with the symbol $\mathcal{B}$ a complex  Banach space with norm $\| \; \|$, and with $\mathfrak{B}(\mathcal{B})$ the set of bounded linear operators $L:\mathcal{B}\to \mathcal{B}$. Finally $\mathbb{R}_+:= [0,+\infty)$.  We recall that a mapping $V:\mathbb{R}_+\to \mathfrak{B}(\mathcal{B})$, is called a {\em $C_0$-semigroup}, or  a {\em strongly continuous one-parameter semigroup} (of  bounded operators) if it satisfies the following conditions:
 \begin{itemize}
\item $V(0)= I$ the identity operator on $\mathcal{B}$,
	
\item   $V(t+s)=V(t)V(s)$ if $t,s \in \bR_+$,
	
\item  $\bR_+ \ni t \mapsto V(t)x$ is continuous for every $x\in \mathcal{B}$, i.e., $V$ is continuous in the {\em strong operator topology}.\hfill $\blacksquare$
	\end{itemize}
	
	

As is well known \cite{EN1}, if $(V(t))_{t\geq 0}$ is a $C_0$-semigroup in Banach space $\mathcal{B}$, then the set
 \begin{equation}\nonumber D(L) := \left\{\varphi\in \mathcal{B} \: \left|\: \exists \lim_{t\to +0}\frac{V(t)\varphi-\varphi}{t}\right.\right\}
\label{DL}\end{equation} is a dense linear subspace of $\mathcal{B}$ invariant under the action of each $V(t)$, $t\geq 0$.
 The operator $L: D(L)\to \mathcal{B}$  $$L\varphi=\lim_{t\to +0}\frac{V(t)\varphi-\varphi}{t}\:, \quad \varphi \in D(L)$$
 is called the ( infinitesimal) generator of the $C_0$-semigroup $V$.
 The generator turns out to be  a closed linear operator that defines $V$ uniquely which, in turn, is denoted\footnote{As is well
 known, this notation is only formal in the general case even if in some situations it has a rigorous 
meaning in terms of norm-converging series if $L$ is bounded respectively spectral functional calculus in Hilbert spaces when $L$ is normal.} as $V(t)=e^{tL}$.

If $L: D(L) \to \mathcal{B}$ with $D(L) \subset \mathcal{B}$ is an operator, the problem of finding a function $u\colon \bR_+\to \mathcal{B}$ such that 
\begin{equation}\label{ACP1}
\left\{ \begin{array}{ll}
 \frac{d}{d t}u(t)= Lu(t); & t\geq 0,\\
 u(0)=u_0,\\
  \end{array} \right.
\end{equation}
is called the abstract Cauchy problem (for the evolution equation) associated
to $L$. A function $u\colon \bR_+\to \mathcal{B}$ is called a classical solution to the abstract Cauchy problem (\ref{ACP1})
if, for every $t\geq 0$, the function $u$ has a continuous derivative (in the topology of $\mathcal{B}$) $u'\colon \bR_+\to \mathcal{B}$, it holds $u(t)\in D(L)$ for $t\in \bR_+$, and (\ref{ACP1}) holds. 

In the case where the operator $L: D(L) \to \mathcal{B}$ in \eqref{ACP1} is the generator of a strongly continuous semigroup $(V(t))_{t\geq 0}$ in the Banach space $\mathcal{B}$, then it is well known (see, e.g., Proposition 6.2 in \cite{EN1}, p. 145) that for every $u_0\in D(L)$ there is a unique classical solution to the abstract Cauchy problem (\ref{ACP1}), which is given by formula $u(t)=V(t)u_0$.



For later convenience, we recall here Chernoff's theorem \cite{Cher,EN1}.
The latter is a fundamental result in semigroup theory that provides a method for approximating strongly continuous semigroups in terms of the limit of suitable operator-valued functions. 
\begin{theorem}[Chernoff theorem]\label{Chernoff-theorem}
Let $(e^{tL})_{t\geq 0}$ be a $C_0$-semigroup on a complex Banach space $\mathcal{B}$ with generator  $L: D(L) \subset \mathcal{B}\to \mathcal{B}$ and let $S:\mathbb{R}_+\to \mathfrak{B}(\mathcal{B})$ be a map satisfying the following conditions:
\begin{enumerate}
    \item[i.]\label{Item: Chernoff conditions - normalization}
    $S(0) = I$, i.e., $S(0)f=f$ for every $f \in \mathcal{B}$;
    \item[ii.]\label{Item: Chernoff conditions - norm bound}
    There exists $\omega\in\mathbb{R}$ such that $\|S(t)\|\leq e^{\omega t}$ for all $t\geq 0$;
    \item[iii.]\label{Item: Chernoff conditions - strong continuity}
    The function $S$ is  continuous in the strong topology in  $\mathfrak{B}(\mathcal{B})$;
    \item[iv.]\label{Item: Chernoff conditions - approximation on a core}
    There exists a linear subspace $\mathcal{D} \subset D(L)$ that is a core for the operator $L $   and such that  $\lim_{t \to 0}(S(t)f-f-tLf)/t=0$ for each $f \in \mathcal{D}$.
\end{enumerate}
Then the following holds:
\begin{equation}\label{Chest}
\lim_{n\to\infty}\sup_{t\in[0,T]}\left\|S(t/n)^nf - e^{tL}f\right\|=0,\quad \mbox{for every $f\in \mathcal{B}$ and every $T>0$,}
\end{equation}
where $S(t/n)^n$ is a composition of $n$ copies of the linear bounded operator $S(t/n)$.
\end{theorem}
  

A map $S:\bR_+\to \mathfrak{B}(\mathcal{B})$  satisfying formula \eqref{Chest} is called a {\em Chernoff function} for operator $L$ or is said to be {\em Chernoff-equivalent} to the $C_0$-semigroup $(e^{tL})_{t\geq 0}$, while the expression $S(t/n)^nf$ is called a {\em Chernoff approximation}  for $e^{tL}f$. 
\begin{remark}\label{remark Classical Solution}
    If $u_0\in D(L)$ the $\mathcal{B}$-valued function $$u(t):=\lim_{n\to\infty}S(t/n)^nu_0=e^{tL}u_0$$ is the classical solution of the Cauchy problem (\ref{ACP1}), hence Chernoff approximation expressions become approximations to the solution with respect to the norm in $\mathcal{B}$.
\end{remark}
\subsection{The Heisenberg group}\label{subsection-Heisenberg-group}

The Heisenberg group is the Lie group $\mathbb{H}^d:=\mathbb{R}^{2d}\times\mathbb{R}$ whose group law reads
\begin{align}\label{Eq: Heisenberg group law}
	(x,y,s)\cdot(x',y',s')
	:=(x+x',y+y',s+s'+x'y-xy')\,,
\end{align}
where $x,x',y,y'\in\mathbb{R}^d$ while $s,s'\in\mathbb{R}$.
\ndnote{Strange $\cdot_\sigma$ needed to avoid mess with $\mathbb{C}$-rotation.}
Occasionally we will adopt the short notation $(x,y,s)=(z,s)$, $z\in\mathbb{R}^{2d}$ and set {$z\cdot_\sigma z':=x'y-xy'$}.
The generators of the corresponding Lie algebra $\mathfrak{h}^d:=\operatorname{Lie}(\mathbb{H}^d)$ are given by
\begin{align}\label{basis vectors}
    X_i=\partial_{x^i}+y^i\partial_s\,,
    \qquad
    Y_i=\partial_{y^i}-x^i\partial_s\,,
    \qquad
    T=\partial_s\,,
\end{align}
where $i\in\{1,\ldots,d\}$. The only non-trivial commutator relations being
$$[X_j,Y_j]=-2T,\qquad j=1,\dots,d\,.$$
Thus the vector fields
$X_1,\dots,X_d,Y_1\dots,Y_d$ satisfy  H\"ormander's rank condition \cite{BLU07},
and $\mathbb H^d$ is a step $2$ Carnot group, the stratification of the
Lie algebra of left invariant vector fields being given by $\mathfrak h^d=V_1\oplus V_2$ 
$$V_1=\mathrm{span}\;\{X_1,\dots,X_d,Y_1,\dots,Y_d\} \quad \mbox{and}
\quad V_2=\mathrm{span}\;\{T\}.
$$ 
We recall that an absolutely continuous curve $\gamma\colon [a,b]\to \mathbb{H}^d$ is horizontal if there exist $u_{1}, \ldots, u_{2d}\in L^{1}[a,b]$ such that for almost every $t\in [a,b]$:
\[\gamma'(t)=\sum_{i=1}^{d}\left(u_{i}(t)X_{i}(\gamma(t))+u_{i+d}(t)Y_{i}(\gamma(t))\right).\]
Define the horizontal length of such a curve $\gamma$ by $L(\gamma)=\int_{a}^{b}|u(t)|\mathrm{d} t$, where $u=(u_{1}, \ldots, u_{2d})$ and $|\,\cdot\,|$ denotes the Euclidean norm on $\mathbb{R}^{2d}$.
The Carnot-Carath\'eodory distance (CC distance) between points $x, y \in \mathbb{H}^d$ is defined by:
\[d_{\mathbb{H}}(x,y)=\inf \{ L(\gamma) : \gamma \colon [0,1]\to \mathbb{H}^d \mbox{ horizontal joining}\ x\mbox{ to}\ y \}.\]
The Chow-Rashevskii Theorem implies that any two points of $\mathbb{H}^d$ can be connected by horizontal curves \cite[Theorem 9.1.3]{BLU07}. It follows that the CC distance is indeed a distance on $\mathbb{H}^d$. 
For brevity we write $\|x\|$ instead of $d_{\mathbb{H}}(x,0)$. It is well known that $d_{\mathbb{H}}$ induces the same topology as the Euclidean distance but the two are not Lipschitz equivalent.\\
It is also well known that the $(2d+1)$-dimensional Lebesgue measure $\mathcal{L}^{2d+1}$ on $\mathbb{H}^d\equiv \mathbb{R}^{2d+1}$ is the so-called Haar measure of the group.
Unless otherwise stated, when
saying that a property holds a.e.\ on $\mathbb{H}^d$, we will always mean that it holds a.e.\ with
respect to $\mathcal{L}^{2d+1}$.
In what follows we will denote by 
\begin{align}\label{Eq: Casimir operator}
	L_0:=\sum_{i=1}^d(X_i^2+Y_i^2),
\end{align}
the sub-laplacian of $\mathbb{H}^d$ defined on the space of smooth and compactly supported functions $C^{\infty}_c(\mathbb{H}^d)$. It is easy to verify that $L_0$ is symmetric  with respect to the inner  product of the Hilbert space $L^2(\mathbb{H}^d)=L^2(\mathbb{R}^{2d+1})$. Since the distribution generated by $\{X_i,Y_i\}$ is Lie-bracket generating, then $L_0$ is hypoelliptic (and indeed subelliptic). Moreover, $L_0$ is essentially self-adjoint on $C^{\infty}_c(\mathbb{H}^d)$ \cite{Str86}.  In the following we shall denote $L=\bar L_0$ the unique selfadjoint extension, which  is the counterpart of the Laplacian on $\mathbb{R}^d$. 
In addition, it is straightforward to verify that the space of Schwartz functions $S(\mathbb{H}^d)=S(\mathbb{R}^{2d+1})$ is a core for $L$ 
and if $\psi\in S(\mathbb{H}^d)$ an easy computation gives
\begin{align}\label{formalap}
L\psi(z,s)=(\Delta+\|z\|^2\partial_{ss}+2\sum_{i=1}^d (z_{i+d}\partial_{z_is}-z_i\partial_{z_{i+d}s}))\psi(z,s)\,.
\end{align}

 We shall also consider the operator $L_0$  defined by \eqref{Eq: Casimir operator} on the Banach space $C_0(\bH^d)$ of continuous functions vanishing at $\infty$. This functional setting is particularly relevant for the construction of a corresponding Feller semigroup associated to a diffusion process known as the {\em Brownian motion on $\bH^d$}\cite{App}. This is constructed by considering a standard Brownian motion $B(t)=(B_1^1(t),B_1^2(t),\ldots, B_d^1(t), B_d^2(t))$ in $\mathbb{R}^{2d}$ starting at $0$  and the corresponding  {\em Levì Area integral}
\begin{equation}
S(t)=\sum_{j=1}^d\int_{0}^t B_j^2(s)dB_j^1(s)-B_j^1(s)dB_j^2(s)
\end{equation} and by defining the $\bH^d$-valued process $W$ as $W(t)=(B(t),S(t))$. More generally, considered $\xi \in \mathbb{H}^d$,  the \emph{horizontal Brownian motion} starting at $\xi$ is defined as $W_\xi(t):=\xi\cdot W(t)$, we address the interested reader to \cite{Bau,Gor,Neu} for a detailed introduction.

With an abuse of notation, in the following we shall still denote with $\frac{L}{2}$, where $L:D(L)\subset C_0(\bH^d)\to  C_0(\bH^d) $, the generator of the Feller semigroup, heuristically written as $V(t)=e^{L\frac{t}{2}}$. Furthermore,
the subspace $C_c^\infty(\bH^d) $ is still a core for $L $, see e.g. \cite{CarGor}.

In the sequel, we will consider modifications of the operator $L$  by adding a zero-order term in a certain class of bounded functions, namely 
\begin{equation}\label{operator2}
   L_c=\frac{L}{2}+c\,.
\end{equation}
with $c\in  L^\infty(\bH^d)$ in the case the semigroup is defined on $L^2(\bH^d)$, while $c\in  C_b(\bH^d)$  in the case we consider a Feller semigroup on $C_0(\bH^d)$.
In the first case, by Kato-Rellich theorem (\cite{ReSi}, Th. X.12) the operator $L_c$ is selfadjoint on $D(L)$ and essentially self-adjoint on any core of $L$. In the second case,  since the multiplication operator associated to the map $c\in C_b(\bH^d)$ is bounded, by Corollary 7.2 in \cite{EthKur} the operator sum $L_c:D(L)\subset C_0(\bH^d)\to  C_0(\bH^d) $ generates a strongly continuous semigroup $e^{tL_c}$ satisfying $\|e^{tL_c}\|_{\mathfrak{B}(C_0(\mathbb{H}^d))}\leq e^{t\|c\|_{C_b(\mathbb{H}^d)}}$. Moreover, any core for $L$ is also a core for $L_c$.


\section{Heat equation}
In this section we will define a Chernoff function for the semigroup generated by the operator $L$ defined in \eqref{operator2}.
We will be interested in the action of the latter semigroup either on $L^2(\mathbb{H}^d)$ or on $C_0(\mathbb{H}^d)$.
To this avail we consider the operator $S(\tau)$, $\tau\geq 0$, defined by
\begin{align}\label{Eq: Chernoff approximation for Euclidean Feynman integral}
	[S(\tau)\psi](z,s)
:=\int_{\mathbb{R}^{2d}}\psi(z+\sqrt{\tau}\zeta,s+\sqrt{\tau} z\cdot_\sigma\zeta)\mathrm{d}\mu_G(\zeta)+\tau c(x)\psi(x)\,,
\end{align}
where $\psi\in S(\mathbb{H}^d)$ while $\mathrm{d}\mu_G(\zeta)=\frac{1}{(2\pi)^d}e^{-|\zeta|^2/2}\mathrm{d}^{2d}\zeta$ is the standard Gaussian measure.
The function $c\colon\mathbb{H}^d\to\mathbb{R}$ will be chosen in a suitable space related to the Banach space of interest.
Specifically, we will assume that $c\in L^\infty(\mathbb{H}^d)$ when dealing with the semigroup in $L^2(\mathbb{H}^d)$, while we will assume that $c\in C_b(\mathbb{H}^d)$ when considering the semigroup in $C_0(\mathbb{H}^d)$.
Our goal is to prove that $S(\tau)$ fulfills the hypothesis of Theorem \ref{Chernoff-theorem}, thus $S$ is a Chernoff function for the operator defined in \eqref{operator2}.

\begin{proposition}\label{Prop: Chernoff approximation for Euclidean Feynman integral - bound norm and strong continuity}
	Let $c\in L^\infty(\bH^d)$. Then the operator defined in \eqref{Eq: Chernoff approximation for Euclidean Feynman integral} extends to a linear and continuous operator on  $L^2(\mathbb{H}^d)$. Furthermore,
	\begin{align*}
\lim_{n\to\infty}\sup_{\tau\in[0,\tau_0]}
		\|S(\tau/n)\psi-e^{tL_c}\psi\|_{L^2(\mathbb{H}^d)}=0\,,
	\end{align*}
	for all $\psi\in L^2(\mathbb{H}^d)$ and $\tau_0>0$ where $L_c$ is as in \eqref{operator2}.
\end{proposition}

\begin{proof}
Let us first assume $c\equiv 0$.	We start by proving that $S(\tau)$ can be extended to an element in $\mathfrak{B}(L^2{(\mathbb{H}^d)})$.
	Let $\psi\in S(\mathbb{H}^d)$. By Jensen inequality 
    we find
	\begin{align*}
		&\int_{\mathbb{R}^{2d+1}}|[S(\tau)\psi](z,s)|^2\mathrm{d}s\mathrm{d}^{2d}z
		\\
		&=\int_{\mathbb{R}^{2d+1}}\bigg|\int_{\mathbb{R}^{2d}}\psi(z+\sqrt{\tau}\zeta,s+\sqrt{\tau} z\cdot_\sigma\zeta)
		\mathrm{d}\mu_G(\zeta)\bigg|^2
		\mathrm{d}s\mathrm{d}^{2d}z
		\\
		&\leq\int_{\mathbb{R}^{2d+1}}\int_{\mathbb{R}^{2d}}|\psi(z+\sqrt{\tau}\zeta,s+\sqrt{\tau} z\cdot_\sigma\zeta)|^2
		\mathrm{d}\mu_G(\zeta)
		\mathrm{d}s\mathrm{d}^{2d}z
		\\
		&=\int_{\mathbb{R}^{2d+1}}\int_{\mathbb{R}^{2d}}|\psi(z,s)|^2
		\mathrm{d}\mu_G(\zeta)
		\mathrm{d}s\mathrm{d}^{2d}z\\
  &=\|\psi\|_{L^2(\mathbb{H}^d)}^2\,,
	\end{align*}
	where we used the translation invariance of the Lebesgue measure on $\mathbb{R}^{2d+1}$ together with Fubini-Tonelli's Theorem.
	This proves that $S(\tau)$ is well-defined on $S(\mathbb{H}^d)$, moreover, it can be extended by density to a bounded, linear operator $S(\tau)\in\mathfrak{B}(L^2(\mathbb{H}^d))$. By direct computation, $S(0)=I$ and $\|S(\tau)\|_{\mathfrak{B}(L^2(\mathbb{H}^d))}\leq 1$, hence condition \ref{Item: Chernoff conditions - normalization} and \ref{Item: Chernoff conditions - norm bound} of Theorem \ref{Chernoff-theorem} are fulfilled
	
	We now prove conditions \ref{Item: Chernoff conditions - strong continuity}) and \ref{Item: Chernoff conditions - approximation on a core}) in Theorem \ref{Chernoff-theorem}.
	
	\begin{description}
		\item[\ref{Item: Chernoff conditions - strong continuity})]
			We need to prove that,
			\begin{align*}
				\|[S(\tau)-S(\tau')]\psi\|_{L^2(\mathbb{H}^d)}
				\underset{\tau\to\tau'}{\longrightarrow}0\,,
			\end{align*}
		for all $\psi\in L^2(\mathbb{H}^d)$.
		For simplicity we consider $\tau'=0$, the general case being treated analogously.
		Let us first consider $\psi_{2d}\in L^2(\mathbb{R}^{2d})$ and $\psi_1\in L^2(\mathbb{R})$. Then
		\begin{align*}
			&\|(S(\tau)-I)(\psi_{2d}\otimes\psi_1)\|_{L^2(\mathbb{H}^d)}^2\\
	&\leq\int_{\mathbb{R}^{2d+1}}\int_{\mathbb{R}^{2d}}
			|\psi_{2d}(z+\sqrt{\tau}\zeta)\psi_1(s+\sqrt{\tau}z\cdot_\sigma\zeta)
			-\psi_{2d}(z)\psi_1(s)|^2
\mathrm{d}\mu_G(\zeta)\mathrm{d}s\mathrm{d}^{2d}z
			\\
	&\leq\int_{\mathbb{R}^{2d+1}}\int_{\mathbb{R}^{2d}}
			|\psi_{2d}(z+\sqrt{\tau}\zeta)(\psi_1(s+\sqrt{\tau}z\cdot_\sigma\zeta)-\psi_1(s))|^2
			\mathrm{d}\mu_G(\zeta)\mathrm{d}s\mathrm{d}^{2d}z
			\\
	&+\int_{\mathbb{R}^{2d+1}}\int_{\mathbb{R}^{2d}}
			|(\psi_{2d}(z+\sqrt{\tau}\zeta)-\psi_{2d}(z))\psi_1(s)|^2
			\mathrm{d}\mu_G(\zeta)\mathrm{d}s\mathrm{d}^{2d}z
			\\
		&=\int_{\mathbb{R}^{2d}}\int_{\mathbb{R}^{2d}}
			|\psi_{2d}(z)|^2
			\|(T^{(1)}_{\sqrt{\tau}(z-\sqrt{\tau}\zeta)\cdot_\sigma\zeta}-I)\psi_1\|_{L^2(\mathbb{R})}^2
			\mathrm{d}\mu_G(\zeta)\mathrm{d}s\mathrm{d}^{2d}z
			\\
			&+\int_{\mathbb{R}^{2d}}
			\|(T^{(2d)}_{\sqrt{\tau}\zeta}-I)\psi_{2d}\|^2_{L^2(\mathbb{R}^{2d})}
			\|\psi_1\|^2_{L^2(\mathbb{R})}
			\mathrm{d}\mu_G(\zeta)\,,
		\end{align*}
		where $T^{(2d)}_\zeta$, $\zeta\in\mathbb{R}^k$ denotes the unitary group of translations on $L^2(\mathbb{R}^{2d})$ and similarly for $T^{(1)}_{s}$, $s\in\mathbb{R}$.
		Since the latter operators are strongly continuous on the corresponding $L^2$ - spaces, the integrands
		\begin{align*}
			\|(T^{(2d)}_{\sqrt{\tau}\zeta}-I)\psi_{2d}\|^2_{L^2(\mathbb{R}^{2d})}\,,
			\qquad
			\|(T^{(1)}_{\sqrt{\tau}(z-\sqrt{\tau}\zeta)\cdot_\sigma\zeta}-I)\psi_1\|_{L^2(\mathbb{R})}^2\,,
		\end{align*}
		vanish as $\tau\to 0^+$ for almost all $z,\zeta\in\mathbb{R}^{2d}$, moreover, each term is uniformly bounded:
		\begin{align*}
			&\|(T^{(1)}_{\sqrt{\tau}\zeta}-I)\psi_{2d}\|^2_{L^2(\mathbb{R}^{2d})}
			\leq 2\|\psi_{2d}\|^2_{L^2(\mathbb{R}^{2d})}\,,
			\\
			&\|(T^{(2d)}_{\sqrt{\tau}(z-\sqrt{\tau}\zeta)\cdot_\sigma\zeta}-I)\psi_1\|_{L^2(\mathbb{R})}^2
			\leq 2\|\psi_1\|_{L^2(\mathbb{R})}^2\,.
		\end{align*}
		Thus, by the dominated convergence theorem we have
		\begin{align*}
			\|(S(\tau)-I)(\psi_{2d}\otimes\psi_1)\|_{L^2(\mathbb{H}^d)}^2
			\underset{\tau\to 0^+}{\longrightarrow}0\,.
		\end{align*}
		This proves that $S(\tau)$ is strongly continuous on $L^2(\mathbb{R}^{2d})\otimes L^2(\mathbb{R})$, which is dense in $L^2(\mathbb{H}^d)$.
		The uniform bound on the operator norm of $S(\tau)$ allows to extend this result to $L^2(\mathbb{H}^d)$.
		\item[\ref{Item: Chernoff conditions - approximation on a core})]
		We will prove Condition \ref{Item: Chernoff conditions - approximation on a core}) on $S(\mathbb{H}^d)$, which is a core for $L$.
		To this avail, let $\psi\in S(\mathbb{H}^d)$. Then
		\begin{multline*}
			\psi(z+\sqrt{\tau}\zeta,s+\sqrt{\tau}z\cdot_\sigma\zeta)
			=\psi(z,s)
			+\sqrt{\tau}[\zeta\partial_z+z\cdot_\sigma\zeta\partial_s]\psi(z,s)
			\\+\frac{\tau}{2}[\zeta\partial_z+z\cdot_\sigma\zeta\partial_s]^2\psi(z,s)
			+\frac{\tau^{3/2}}{6}[\zeta\partial_z+z\cdot_\sigma\zeta\partial_s]^3\psi(z,s)
			\\
			+\frac{1}{6}\int_0^{\sqrt{\tau}}
			\Big(\frac{\mathrm{d}}{\mathrm{d}\tilde s}\Big)^4
			\psi(z+\tilde s\zeta,s+\tilde sz\cdot_\sigma\zeta)
			(\sqrt{\tau}-\tilde s)^3\mathrm{d}\tilde s\,,
		\end{multline*}
		where $\zeta\partial_z=\zeta^i\partial_{z^i}$.
		Integrating the previous expansion with respect to $\mu_G$ leads to
		\begin{align}\label{representation-remainder}
			\int_{\mathbb{R}^{2d}}
\psi(z+\sqrt{\tau}\zeta,s+\sqrt{\tau}z\cdot_\sigma\zeta)
			\mathrm{d}\mu_G(\zeta)
			&=\psi(z,s)
			+\frac{\tau}{2}\int_{\mathbb{R}^{2d}}
			[\zeta\partial_z+z\cdot_\sigma\zeta\partial_s]^2\psi(z,s)
			\mathrm{d}\mu_G(\zeta)
			\\
   \nonumber
			&+\frac{1}{6}\int_{\mathbb{R}^{2d}}\int_0^{\sqrt{\tau}}
			\Big(\frac{\mathrm{d}}{\mathrm{d}\tilde s}\Big)^4
			\psi(z+\tilde s\zeta,s+\tilde sz\cdot_\sigma\zeta)
			(\sqrt{\tau}-\tilde s)^3\mathrm{d}\tilde s\mathrm{d}\mu_G(\zeta).
   \end{align}
   Now we claim that
   \begin{equation}\nonumber
   \int_{\mathbb{R}^{2d}}
			[\zeta\partial_z+z\cdot_\sigma\zeta\partial_s]^2\psi(z,s)
			\mathrm{d}\mu_G(\zeta)=L\psi(z,s).
   \end{equation}
   Indeed,
   \[
[\zeta\partial_z+z\cdot_\sigma\zeta\partial_s]^2=\zeta^i\zeta^j\partial_{z_iz_j}+2\zeta^iz\cdot_\sigma\zeta\partial_{sz_i}+(z\cdot_\sigma\zeta)^2\partial_{ss},
   \]
   Integrating the previous expression with respect to the Gaussian measure, using a symmetry argument and the explicit formulas for the moments of the Gaussian measure we get
   \begin{align*}
   &\int_{\mathbb{R}^{2d}}
			[\zeta\partial_z+z\cdot_\sigma\zeta\partial_s]^2\psi(z,s)\mathrm{d}\mu_G(\zeta)\\
   &=\int_{\mathbb{R}^{2d}}
			[\zeta^i\zeta^j\partial_{z_iz_j}+2\zeta^i z\cdot_\sigma\zeta\partial_{sz_i}+(z\cdot_\sigma\zeta)^2\partial_{ss}]\psi(z,s)\mathrm{d}\mu_G(\zeta)\\
   &=\partial_{z_iz_j}\psi(z,s)\int_{\mathbb{R}^{2d}}\zeta^i\zeta^j\mathrm{d}\mu_G(\zeta)+2\partial_{sz_i}\psi(z,s)\int_{\mathbb{R}^{2d}}\zeta^i z\cdot_\sigma\zeta\mathrm{d}\mu_G(\zeta)+\partial_{ss}\psi(z,s)\int_{\mathbb{R}^{2d}}(z\cdot_\sigma\zeta)^2
			\mathrm{d}\mu_G(\zeta)\\
   &=\partial_{z_iz_i}\psi(z,s)+2(z_{i+d}\partial_{z_is}-z_i\partial_{z_{i+d}s})\psi(z,s)+\|z\|^2\partial_{ss}\psi(z,s)\\
   &=L\psi(z,s)
   \end{align*}
   
		\noindent Overall we may write
		\begin{multline*}
			\left\|
			\frac{S(\tau)-I-\tau L}{\tau}\psi
			\right\|_{L^2(\mathbb{H}^d)}^2
			\\\leq
			\frac{\tau}{6}\int_{\mathbb{R}^{2d+1}}
			\bigg|\int_{\mathbb{R}^{2d}}\int_0^1
			[\zeta\partial_z+z\cdot_\sigma\zeta\partial_s]^4
			\psi(z+\sqrt{\tau} u\zeta,s+\sqrt{\tau} uz\cdot_\sigma\zeta)
			(1-u)^3\mathrm{d}u\mathrm{d}\mu_G(\zeta)\bigg|^2
			\mathrm{d}s\mathrm{d}^{2d}z
			\\
			\leq
			\frac{\tau}{24}\int_{\mathbb{R}^{2d+1}}
			\int_{\mathbb{R}^{2d}}\int_0^1
			|[\zeta\partial_z+z\cdot_\sigma\zeta\partial_s]^4
			\psi(z+\sqrt{\tau} u\zeta,s+\sqrt{\tau} uz\cdot_\sigma\zeta)|^2
			(1-u)^3\mathrm{d}u\mathrm{d}\mu_G(\zeta)
			\mathrm{d}s\mathrm{d}^{2d}z\,,
		\end{multline*}
		where we used again Jensen inequality with respect to the probability measure $4(1-u)^3\mathrm{d}u\mathrm{d}\mu_G(\zeta)$.
		We observe that if $\tau\in(0,1)$, then
		\begin{align*}
			\int_{\mathbb{R}^{2d+1}}
			|[\zeta\partial_z+z\cdot_\sigma\zeta\partial_s]^4
			\psi(z+\sqrt{\tau} u\zeta,s+\sqrt{\tau} uz\cdot_\sigma\zeta)|^2
			\mathrm{d}s\mathrm{d}^{2d}z
		\end{align*}
		can be bounded by a $\tau$-independent polynomial function in $u,\zeta$, which is integrable with respect to $4(1-u)^3\mathrm{d}u\mathrm{d}\mu_G(\zeta)$.
		Thus, it follows that
		\begin{align*}
			\left\|
			\frac{S(\tau)-I-\tau \frac{L}{2}}{\tau}\psi
			\right\|_{L^2(\mathbb{H}^d)}^2
			=O(\tau)\,,
		\end{align*}
		proving condition \ref{Item: Chernoff conditions - approximation on a core}).
	\end{description}
Let us now assume $c\neq 0$. Let $S_0$ denote the Chernoff function of $L$ with $c\equiv 0$ and let $S$ denote the analog for the case $c\neq 0$. If $\psi\in L^2(\mathbb{H}^d)$, we have $S(\tau)\psi=S_0(\tau)\psi+\tau c \psi\in L^2(\mathbb{H}^d)$ because $S_0(\tau)\psi \in L^2(\mathbb{H}^d)$, $\psi \in L^2(\mathbb{H}^d)$ and $c\in L^\infty(\mathbb{H}^d)$. This proves that $S$ is an operator on $L^{2}(\mathbb{H}^d)$, moreover $S$ is continuous because it is the sum of two continuous operators. Condition \ref{Item: Chernoff conditions - normalization}) is clearly satisfied. Now, 
\begin{align*}
\|S(\tau)\psi\|_{L^2(\mathbb{H}^d)}=\|S_0(\tau)\psi+\tau c \psi\|_{L^2(\mathbb{H}^d)}&\leq \|S_0(\tau)\|_{\mathfrak{B}(L^2(\mathbb{H}^d))}\|\psi\|_{L^2(\mathbb{H}^d)}+|\tau|\|c\psi\|_{L^2(\mathbb{H}^d)}
\\
&\leq (1+\tau{\|c\|_{L^\infty(\mathbb{H}^d)})}\|\psi\|_{L^2(\mathbb{H}^d)}\\
&\leq e^{\tau\|c\|_{L^\infty(\mathbb{H}^d)}}\|\psi\|_{L^2(\mathbb{H}^d)}.
\end{align*}
Finally, let $\psi\in C^{\infty}_c(\mathbb{H}^d)$, then exploiting Condition \ref{Item: Chernoff conditions - approximation on a core} in Theorem \ref{Chernoff-theorem} for $c=0$ we get
\begin{align*}
S_0(\tau)\psi=\psi+\tau \frac{L}{2}\psi + o(\tau)
\end{align*}
therefore
\begin{align*}
S(\tau)\psi=\psi+\tau L_c\psi + o(\tau)
\end{align*}
thus proving Condition \ref{Item: Chernoff conditions - approximation on a core} in Theorem \ref{Chernoff-theorem}.
\end{proof}
In the following theorem we show that Proposition \ref{Prop: Chernoff approximation for Euclidean Feynman integral - bound norm and strong continuity} continues to hold by replacing $L^2(\mathbb{H}^d)$ with $C_0(\mathbb{H}^d)$. The result will be applied in the section \ref{Sec: Construction of the Brownian motion on Hd as weak limit of random walks} to provide a new construction of the Brownian motion in the Heisenberg group.

\begin{theorem}\label{continuous}
Let $c\in C_b(\mathbb{H}^d)$. Then the operator defined in \eqref{Eq: Chernoff approximation for Euclidean Feynman integral} is such that for all $\tau\geq 0$ $S(\tau)(C_0(\mathbb{H}^d))\subset C_0(\mathbb{H}^d)$. 
	Furthermore,
	\begin{align}\label{convergence-cherrnoff-C0}
		\lim_{n\to\infty}\sup_{\tau\in[0,\tau_0]}
		\|S(\tau/n)\psi-e^{tL_c}\psi\|_{C_0(\mathbb{H}^d)}=0\,,
	\end{align}
	for all $\psi\in C_0(\mathbb{H}^d)$ and $\tau_0>0$ where $L_c$ is as in \eqref{operator2}.
\end{theorem}
\begin{proof}

       Let us begin by assuming \( c = 0 \).  
The continuity of the mapping \( (z,s) \mapsto [S(\tau)\psi](z,s) \) follows from the continuity of the function \( (z,s) \mapsto \psi(z + \sqrt{\tau} \zeta, s + \sqrt{\tau} z \cdot_\sigma \zeta) \) and the dominated convergence theorem, given that \( \psi \in C_0(\mathbb{H}^d) \) by assumption.\\
    Let us now show that the function \( S(\tau)\psi \) vanishes at infinity.  
Given \( \epsilon > 0 \), let \( R > 0 \) be such that \( \mu_G(B_R(0)^c) < \frac{\epsilon}{2\|\psi\|_\infty} \), where \( B_R(0) \) denotes the ball in \( \mathbb{R}^{2d} \) centered at \( 0 \) with radius \( R \).  
We then obtain:  
\[
\big|[S(\tau)\psi](z,s)\big| \leq \int_{B_R(0)} \big|\psi(z+\sqrt{\tau} \zeta, s+\sqrt{\tau} z\cdot_\sigma\zeta)\big| \,\mathrm{d}\mu_G(\zeta) + \frac{\epsilon}{2}.
\]
Since \( \psi \in C_0(\mathbb{H}^d) \), there exists a compact set \( K_\epsilon \subset \mathbb{R}^{2d+1}\) such that \( |\psi(z,s)| < \frac{\epsilon}{2} \) for all \( (z,s) \in K_\epsilon^c \).
    Let us now consider the set $A\subset \bR^{2d+1}$ defined as:
    \begin{equation}
        A:=\{(z,s)\in \bR^{2d+1}\colon \forall \zeta \in  B_R(0), \, (z+\sqrt{\tau}\zeta,s+\sqrt{\tau} z\cdot_\sigma\zeta)\in K_\epsilon^c \}
    \end{equation}
    and its complementary $B:=A^c$ given by
    \begin{equation}
        B=\{(z,s)\in \bR^{2d+1}\colon \exists \zeta \in  B_R(0),\, (z+\sqrt{\tau}\zeta,s+\sqrt{\tau} z\cdot_\sigma\zeta)\in K_\epsilon \}.
    \end{equation}
    Clearly, $B$ is a bounded set, hence its closure $K_{\epsilon,\tau}:=\overline{B}$ is compact and by definition of $B$ we have that for any $(z,s)\in \overline{B}^c\subset A$ we have 
    \[|\psi(z+\sqrt{\tau}\zeta,s+\sqrt{\tau} z\cdot_\sigma\zeta)|<\epsilon \quad \forall \zeta\in B_R(0)\,, \]
    and we eventually get
    \[|[S(\tau)\psi](z,s)|<\epsilon \quad \forall (z,s)\in  \overline{B}^c\,.\]
    
   
   In order to prove \eqref{convergence-cherrnoff-C0} it is sufficient to show that the map $S\colon\mathbb{R}_+\to\mathfrak{B}(C_0(\mathbb{H}^d))$ satisfies the four hypothesis of Theorem \ref{Chernoff-theorem}.\\
    First of all, the condition $S(0)=I$ as well as the inequality $\|S(\tau)\|_{\mathfrak{B}(C_0(\mathbb{H}^d))}\leq 1$ for all $\tau\geq 0$ follow by direct computation. Let us now prove that the map $\tau \mapsto S(\tau) $ is strongly continuous. Let us fix a $\tau_0\in \bR^+$ and consider for a generic $\psi \in C_0(\mathbb{H}^d)$, then
    \begin{align*}
        &
        \|S(\tau)\psi -S(\tau_0)\psi\|_{C_0(\mathbb{H}^d)}\\&=\sup_{(z,s)\in \bR^{2d+1}}|S(\tau)\psi(z,s) -S(\tau_0)\psi(z,s)|\\
        &\leq \sup_{(z,s)\in \bR^{2d+1}}\int_{\bR^{2d+1}}|\psi(z+\sqrt{\tau}\zeta,s+\sqrt{\tau} z\cdot_\sigma\zeta)-\psi(z+\sqrt{\tau_0}\zeta,s+\sqrt{\tau_0} z\cdot_\sigma\zeta)|\mathrm{d}\mu_G(\zeta)\,.
    \end{align*}
    Fixed $\epsilon >0$, let us pick $R>0$ such that $\mu_G(B_R(0)^c)<\frac{\epsilon}{4\|\psi\|}$, in such a way that 
    \begin{multline*}
        \int_{\bR^{2d+1}}|\psi(z+\sqrt{\tau}\zeta,s+\sqrt{\tau} z\cdot_\sigma\zeta)-\psi(z+\sqrt{\tau_0}\zeta,s+\sqrt{\tau} z\cdot_\sigma\zeta)|\mathrm{d}\mu_G(\zeta)\\
        \leq\int_{B_R(0)}|\psi(z+\sqrt{\tau}\zeta,s+\sqrt{\tau} z\cdot_\sigma\zeta)-\psi(z+\sqrt{\tau_0}\zeta,s+\sqrt{\tau} z\cdot_\sigma\zeta)|\mathrm{d}\mu_G(\zeta)+\frac{\epsilon}{2}.
    \end{multline*}
    Further, let us consider the compact set $K_\epsilon\subset \bR^{2d+1}$ such that $|\psi(z,s)|<\epsilon /4$ for all $(z,s)\in K_\epsilon^c$. Considered the set $A_{\epsilon,\tau_0}\subset \bR^{2d+1}$ defined as
    $$A_{\epsilon,\tau_0}:=\{(z,s)\in \bR^{2d+1}\colon (z+\sqrt{\tau}\zeta,s+\sqrt{\tau} z\cdot_\sigma\zeta)\in K_\epsilon^c\,, \forall \zeta \in B_R(0), \tau\in [0\wedge (\tau_0-1), \tau_0+1]\}$$
    we have that 
    $$\int_{B_R(0)}|\psi(z+\sqrt{\tau}\zeta,s+\sqrt{\tau} z\cdot_\sigma\zeta)-\psi(z+\sqrt{\tau_0}\zeta,s+\sqrt{\tau} z\cdot_\sigma\zeta)|\mathrm{d}\mu_G(\zeta)<\epsilon/2$$
    for all $(z,s)\in A_{\epsilon,\tau_0}$. Moreover, the set $A_{\epsilon,\tau_0}^c$ given by:
    $$A_{\epsilon,\tau_0}^c:=\{(z,s)\in \bR^{2d+1}\colon \exists \zeta \in B_R(0), \tau\in [0\wedge (\tau_0-1), \tau_0+1] \,, (z+\sqrt{\tau}\zeta,s+\sqrt{\tau} z\cdot_\sigma\zeta)\in K_\epsilon\}$$
    is bounded by construction. Hence, the subset $B\subset \bR^{2d+1}$ of points $(z,s) \in  \bR^{2d+1}$ of the form
    $$(z,s)=(z'+\sqrt{\tau}\zeta,s'+\sqrt{\tau} z'\cdot_\sigma\zeta)$$
    for some $(z',s')\in A_{\epsilon,\tau_0}^c $, $\zeta\in B_R(0)$ and $\tau\in  [0\wedge (\tau_0-1), \tau_0+1] $ is bounded, and its closure $\bar B$ is compact. 
    If we now consider a generic point $(z,t)\in A_{\epsilon,\tau_0}^c$, by the uniform continuity of the map $\psi:\bar B\to \bC$ and the continuity of the map $\tau\mapsto (z+\sqrt{\tau}\zeta,s+\sqrt{\tau} z\cdot_\sigma\zeta)$, it is now possible to show that there exists a $\delta>0  $
such that if $\tau\in [0\wedge (\tau_0-1), \tau_0+1] $ and $|\tau-\tau_0|<\delta$ we have  $$|\psi(z+\sqrt{\tau}\zeta,s+\sqrt{\tau} z\cdot_\sigma\zeta)-\psi(z+\sqrt{\tau_0}\zeta,s+\sqrt{\tau} z\cdot_\sigma\zeta)|<\epsilon/2$$ hence 
$|S(\tau)\psi (z,s)-S(\tau_0)\psi(z,s)|<\epsilon$.\\
We can now prove condition \ref{Item: Chernoff conditions - approximation on a core} of theorem \ref{Chernoff-theorem} for $\psi \in S(\mathbb{H}^d)$, which is a core for $L$. By Eq. \eqref{representation-remainder} we have:
\begin{equation*}
S(\tau)\psi(z,s)-\psi(z,s)-\tau \frac{L}{2}\psi (z,s)= 
    \frac{\tau^2}{6}r_\tau(z,s)\end{equation*}
    where 
    \begin{align*}\nonumber
   r_\tau(z,s)&=
    \int_{\mathbb{R}^{2d}}\int_0^1
			[\zeta\partial_z+z\cdot_\sigma\zeta\partial_s]^3
			\psi(z+\sqrt{\tau} u\zeta,s+\sqrt{\tau} uz\cdot_\sigma\zeta)
			(1-u)^3\mathrm{d}u\mathrm{d}\mu_G(\zeta)\\
   &=\sum_{\mathbb{\alpha}}\int_{\mathbb{R}^{2d}}\int_0^1
			 P_{\mathbb{\alpha}}(z,s,u,\zeta,\sqrt \tau)
			\psi^{(\mathbb{\alpha})}(z+\sqrt{\tau} u\zeta,s+\sqrt{\tau} uz\cdot_\sigma\zeta)
			\mathrm{d}u\mathrm{d}\mu_G(\zeta)\,. 
   \end{align*}
 where the sum is taken on multindex $\alpha=(\alpha_1,\alpha_2,\alpha_3)\in \bN^3$ such that $\alpha_1+\alpha_2+\alpha_3=3$, while $P_\alpha$ denotes a suitable polynomial function and $\psi^{(\alpha)}$ the $\alpha$-order derivative of $\psi$. By the assumption $\psi\in S(\mathbb{H}^d)$, the constraint $u\in [0,1]$ and the existence of all moments of the Gaussian measure $\mu_G$, it is now straightforward to check that 
 $\|r_\tau\|_{C_0(\mathbb{H}^d)}$ is finite for $\tau $ belonging to an arbitrary bounded interval, hence condition \ref{Item: Chernoff conditions - approximation on a core} of theorem \ref{Chernoff-theorem} is fulfilled.\\
Let us now consider $c\neq 0$. Proceeding as in Proposition \ref{Prop: Chernoff approximation for Euclidean Feynman integral - bound norm and strong continuity} we denote by $S_0$ the Chernoff function of $L$ with $c\equiv 0$ and let $S$ denote the analog for the case $c\neq 0$. If $\psi\in C_0(\mathbb{H}^d)$, we have $S(\tau)\psi=S_0(\tau)\psi+\tau c \psi\in C_0(\mathbb{H}^d)$ because $S_0(\tau)\psi \in C_0(\mathbb{H}^d)$, $\psi \in C_0(\mathbb{H}^d)$ and $c\in C_{b}(\mathbb{H}^d)$. This proves that $S$ is an operator on $C_0(\mathbb{H}^d)$, moreover $S$ is continuous because it is the sum of two continuous operators. Condition \ref{Item: Chernoff conditions - normalization} is clearly satisfied. Now, proceeding exactly as in Proposition \ref{Prop: Chernoff approximation for Euclidean Feynman integral - bound norm and strong continuity} we have 
\begin{align*}
\|S(\tau)\psi\|_{C_0(\mathbb{H}^d)}\leq e^{\tau{\|c\|_{C_b(\mathbb{H}^d)}}}\|\psi\|_{C_0(\mathbb{H}^d)}.
\end{align*}
and for every $\psi\in S(\mathbb{H}^d)$
\begin{align*}
S(\tau)\psi=\psi+\tau L_c\psi + o(\tau)
\end{align*}

thus proving Conditions \ref{Item: Chernoff conditions - norm bound} and  \ref{Item: Chernoff conditions - approximation on a core} in Theorem \ref{Chernoff-theorem}.
\end{proof}
By remark \ref{remark Classical Solution} we get the following result.
\begin{theorem} Under assumptions of Theorem \ref{continuous},  the classical solution of the Cauchy problem 
\begin{equation}\label{CauchyProblemHeat}
 \left\{ \begin{array}{l}
 \frac{\partial }{\partial t}u(t,x)=L_c u(t,x)\\
 u(0,x)=u_0(x)
 \end{array}
 \right.
 \end{equation}
 is given for $u_0\in D(L)$ by \begin{equation*}\label{C-M-1}u(t,x)=\lim_{n\to\infty}(S(t/n)^{n }u_0)(x), \end{equation*} where  $S(t)$  is defined in (\ref{Eq: Chernoff approximation for Euclidean Feynman integral}).
  \end{theorem}

\section{The Schr\"odinger equation on $\bH^{d}$ }
Let us consider the Schr\"odinger equation on $L^2(\bH^{d})$, i.e., the abstract Cauchy problem 
\begin{equation}\label{Schrodinger}
   \begin{cases}
       i\partial_t\psi =H\psi \,, \\
       \psi (0)=\psi_0\,.
   \end{cases} 
\end{equation}
In \eqref{Schrodinger}, $\psi_0\in L^2(\bH^{d})$ is the initial datum, while $H:D(H)\subset L^2(\bH^{d})\to L^2(\bH^{d})$ is the self-adjoint operator defined as the closure of the symmetric operator $H_0:C^\infty_0(\bH^{d})\subset L^2(\bH^d)\to L^2(\bH^{d})$, given by  
\begin{equation*}
    H_0\psi (z,s)=-\frac{L}{2}\psi (z,s) +v(z,s)\psi(z,s)\,, \qquad (z,s)\in \bH^{d}\,,
\end{equation*}
where $L\psi$ is given by Eq. \eqref{formalap}, and $v\in L^\infty(\bH^d)$ is a  bounded potential. As discussed in Section \ref{subsection-Heisenberg-group} for the operator \ref{operator2}, $H_0$ is essentially self-adjoint and in the following we shall denote with the simbol $H=\bar{H}_0$ its closure, i.e. its  unique self-adjoint extension \cite{ReSi,Str86}. Moreover,  
\begin{equation}\label{H}
    D(H)=D(L), \qquad H\psi =-\frac{L}{2}\psi +v\psi\,.
\end{equation}
and $S(\mathbb{H}^d)$ is a core.


Let us denote by $\{U(t)\}_{t\in \bR}$ the strongly continuous group of unitary operators $U(t):L^2(\bH^{d})\to L^2(\bH^{d})$, constructed via functional calculus and the spectral measure of the self-adjoint operator $H$, given by $U(t)=e^{-iHt}$. In particular, if $\psi_0\in D(H)$, then $\psi(t)\equiv U(t)\psi_0$ belongs to $D(H)$ for all $t \in \bR$ and is a classical solution of the abstract Cauchy problem \eqref{Schrodinger}. That is, for all $t\in \bR$, the map $t \mapsto \psi (t)$ has a continuous derivative (in the topology of $L^2(\bH^{d})$), $\partial_t \psi :\bR\to L^2(\bH^{d})$, and satisfies \eqref{Schrodinger}.

\subsection{Costruction of a Chernoff function for the unitary Schr\"odinger group}\label{subsection-Fey1}
Let us consider the map $S\colon\mathbb{R}_+\to\mathfrak{B}(L^2(\mathbb{H}^d)$ defined on $S(\mathbb{H}^d)$ by: 
\begin{align}\label{Eq: Chernoff approximation for Feynman integral - first definition}
	[S(\tau)\psi](z,s)
	:=\frac{1}{(2\pi i)^d}\int^o_{\mathbb{R}^{2d}}e^{i|\zeta|^2/2}\psi(z+\sqrt{\tau}\zeta,s+\sqrt{\tau} z\cdot_\sigma\zeta)\mathrm{d}^{2d}\zeta\,,
\end{align}
where the integral is meant as an oscillatory integral \cite{Hor}. More specifically, the right-hand side of \eqref{Eq: Chernoff approximation for  Feynman integral} is defined as the limit
\begin{equation}\label{limit definition oscillatory integral 1}
    \lim_{\epsilon\downarrow 0} \frac{1}{(2\pi i)^d}\int_{\mathbb{R}^{2d}}\phi(\epsilon \zeta)e^{i|\zeta|^2/2}\psi(z+\sqrt{\tau}\zeta,s+\sqrt{\tau} z\cdot_\sigma\zeta)\mathrm{d}^{2d}\zeta\,,
\end{equation}
where $\phi\in S(\mathbb{H}^d)$ is such that $\phi (0)=1$ and the limit \eqref{limit definition oscillatory integral 1} is required to be independent on the particular choice of $\phi$.
\begin{lemma}\label{lemmaRepS-tau-Schrodinger}
    Let $\psi \in S(\mathbb{H}^d)$. Then the oscillatory integral \eqref{Eq: Chernoff approximation for  Feynman integral} is well defined and can be represented as
\begin{equation}\label{Representation S-tau-Schroedinger case}
   [S(\tau)\psi](z,s)=\int_{\bR^{2d+1}}e^{i(\eta \cdot z +\alpha s)}e^{-\frac{i}{2}\tau |\eta +\alpha \tilde z|^2} \hat \psi(\eta, \alpha)\mathrm{d}^{2d}\eta\mathrm{d}\alpha\,,
\end{equation}
where, for $z=(z_1,z_2)\in\mathbb{R}^{2d}$, $\tilde z:=(z_2,-z_1)$, while $\hat \psi $ is the (Euclidean) Fourier transform of $\psi$, i.e.
\begin{equation*}
\hat \psi(\eta, \alpha)=\frac{1}{(2\pi)^{2d+1}}\int_{\bR^{2d+1}}e^{-i(\eta \cdot z+\alpha s)}\psi (z,s)\mathrm{d}^{2d}z\mathrm{d}s\,.
\end{equation*} 
\end{lemma}
\begin{proof}
Using the Fourier inversion formula 
\begin{equation*}
\psi (z,s)=\int_{\bR^{2d+1}}e^{i(\eta \cdot z+\alpha s)}\hat \psi(\eta, \alpha)\mathrm{d}^{2d}\eta\mathrm{d}\alpha\,,
\end{equation*} 
the limit \eqref{limit definition oscillatory integral 1} becomes:
\begin{equation*}
     \lim_{\epsilon\downarrow 0} \frac{1}{(2\pi i)^d}\int_{\mathbb{R}^{2d}}e^{i|\zeta|^2/2}\phi(\epsilon \zeta)\int_{\bR^{2d+1}}e^{i\eta \cdot (z+\sqrt{\tau}\zeta)}e^{i\alpha (s+\sqrt{\tau} z\cdot_\sigma\zeta)}\hat \psi(\eta, \alpha)\mathrm{d}^{2d}\eta\mathrm{d}\alpha\mathrm{d}^{2d}\zeta\,.
     \end{equation*}
    By Fubini theorem, the latter is equal to:
    \begin{equation*}
     \lim_{\epsilon\downarrow 0} \int_{\bR^{2d+1}} e^{i(\eta \cdot z+\alpha s)}\hat \psi(\eta, \alpha)
     \int_{\mathbb{R}^{2d}}\frac{e^{i|\zeta|^2/2}}{(2\pi i)^d}\phi(\epsilon \zeta)e^{i\sqrt{\tau}(\eta \cdot \zeta+\alpha \tilde z\cdot \zeta)}\mathrm{d}^{2d}\zeta\mathrm{d}^{2d}\eta\mathrm{d}\alpha\,.
     \end{equation*}
     By applying Plancherel's equality to the inner integral we get
     \begin{align*}
    & \lim_{\epsilon\downarrow 0} \int_{\bR^{2d+1}} e^{i(\eta \cdot z+\alpha s)}\hat \psi(\eta, \alpha)
     \int_{\mathbb{R}^{2d}}e^{-\frac{i}{2}|k+\sqrt \tau(\eta +\alpha \tilde z)|^2}\hat \phi(k/\epsilon)\epsilon^{-2d}\mathrm{d}^{2d}k\mathrm{d}^{2d}\eta\mathrm{d}\alpha\\
     =&\lim_{\epsilon\downarrow 0} \int_{\bR^{2d+1}} e^{i(\eta \cdot z+\alpha s)}\hat \psi(\eta, \alpha)
     \int_{\mathbb{R}^{2d}}e^{-\frac{i}{2}|\epsilon k+\sqrt \tau(\eta +\alpha \tilde z)|^2}\hat \phi(k)\mathrm{d}^{2d}k\mathrm{d}^{2d}\eta\mathrm{d}\alpha\,
     \end{align*}
     and by dominated convergence theorem and the condition $\int_{\mathbb{R}^{2d}}\hat \phi(k)\mathrm{d}^{2d}k=\phi(0)=1$ we eventually get the thesis.
     \end{proof}

    Lemma \ref{lemmaRepS-tau-Schrodinger} allows us to show that for any $\tau \in\mathbb{R}_+$, the operator $S(\tau)$ extends to a bounded operator on $L^2(\mathbb{H}^d)$, mapping $S(\mathbb{H}^d)$ into itself.
Indeed, using Equation \eqref{Representation S-tau-Schroedinger case} and expressing $\psi$ in terms of its Fourier transform with respect to the variable $s$, we can represent the action of $S(\tau)$ as the composition of bounded linear operators that map $S(\mathbb{H}^d)$ into itself.
     More specifically,  for $\psi \in S(\mathbb{H}^d)$ Equation \eqref{Representation S-tau-Schroedinger case} yields:
\begin{align}\label{Representation S-tau-Schroedinger case-2}
    S(\tau)\psi(z,\alpha)&=
    e^{-\frac{i}{2}\tau|\alpha \tilde z|^2}
    \int_{\bR^{2d}}e^{i\eta \cdot(z-\tau\alpha\tilde{z})}
    e^{-\frac{i}{2}\tau |\eta|^2}
    \hat \psi(\eta, \alpha)\mathrm{d}^{2d}\eta\\
    &\equiv e^{-\frac{i}{2}\tau|\alpha \tilde z|^2} \Psi_\tau(z-\tau\alpha\tilde{z},\alpha)\,,\nonumber
\end{align}     
where $\Psi_\tau(z,\alpha)$ denotes the $z$-Fourier transform of the map $(\eta,\alpha)\mapsto e^{-\frac{i}{2}\tau |\eta|^2}\hat \psi(\eta, \alpha)\in S(\mathbb{H}^d)$. Hence, by direct inspection, $ S(\tau)\psi\in  S(\mathbb{H}^d)$. Moreover, it is possible to represent $S(\tau): S(\mathbb{H}^d)\to  S(\mathbb{H}^d)$ as the compositions of three unitary operators on $L^2(\mathbb{H}^d)$, i.e. 
\begin{equation}\label{S-tau-representation}
    S(\tau)=U_2(\tau)V(\tau)U_1(\tau)
\end{equation} where:
\begin{itemize}
    \item  $U_1:L^2(\mathbb{H}^d)\to L^2(\mathbb{H}^d)$ is  defined on $\psi\in S(\mathbb{H}^d)$ as
    \begin{equation}\label{U1}
         U_1(\tau)\psi (z,\alpha)=  \int_{\bR^{2d}} e^{i\eta \cdot z}e^{-\frac{i}{2}\tau |\eta|^2}
    \hat \psi(\eta, \alpha)\mathrm{d}^{2d}\eta
    \end{equation}
    which is  unitary, as associated, via Fourier transform, to a unitary multiplication operator.
    \item $U_2(\tau)$ is the  unitary multiplication operator  defined as
    \begin{equation}\label{U2}
        U_2(\tau)\psi (z,\alpha)=e^{-\frac{i}{2}\tau |\alpha \tilde z|^2} \psi(z, \alpha)
    \end{equation} 
    \item $V(\tau)$ is the translation operator defined by
    \begin{equation}\label{V}
        V(\tau)\psi (z,\alpha)=\psi (z-\tau\alpha\tilde{z},\alpha)\, ,
    \end{equation}
    which is bounded on $L^2(\mathbb{H}^d)$ since
    \begin{align*}
        &\int_{\bR^{2d+1}}|\psi(z-\tau\alpha\tilde{z},\alpha)|^2\mathrm{d}^{2d}z\mathrm{d}\alpha\\
        =&\int_{\bR^{2d+1}}\frac{|\psi(z,\alpha)|^2}{1+\alpha^2\tau^2}\mathrm{d}^{2d}z\mathrm{d}\alpha\\
        \leq &\int_{\bR^{2d+1}}|\psi(z,\alpha)|^2\mathrm{d}^{2d}z\mathrm{d}\alpha\,.
    \end{align*}
    \end{itemize}

Hence, for all $\tau\in \bR^+$, $S(\tau) $  can be extended to a unitary operator on $L^2(\mathbb{H}^d)$. Furthermore,  
 the collection $(S(\tau))_{\tau\geq 0}$ fulfills the assumptions of Chernoff theorem \ref{Chernoff-theorem}. 
 \begin{theorem}\label{theorem-Chernoff-approximations-schroedinger-0} Let $L:D(L)\subset L^2(\mathbb{H}^d)\to L^2(\mathbb{H}^d)$ be the horizontal Laplacian and let $\{U(\tau)\}_{\tau\in \bR}$ be the associated strongly continuous  one-parameter group generated by the self-adjoint operator $L/2$, denoted with $U(\tau)=e^{i\tau \frac{L}{2}}$. Let $S:\bR^+\to \mathfrak{B}(L^2(\mathbb{H}^d))$ be the map defined by Eq. \eqref{Eq: Chernoff approximation for Feynman integral - first definition}. Then $S$ is Chernoff equivalent to $(e^{i\tau \frac{L}{2}})_{\tau \geq 0}$    
 \end{theorem}
 \begin{proof}
     The four assumptions of Theorem \ref{Chernoff-theorem} are satisfied, indeed:
 \end{proof}
\begin{enumerate}
    \item[i.] $S(0)=I$, as one can easily check by direct inspection.
    \item[ii.] $\|S(\tau)\|_{\mathfrak{B}(L^2(\mathbb{H}^d))}\leq 1$. This follows directly by the representation formula \eqref{S-tau-representation}, giving $S(\tau)$ as the composition of three unitary operators. 
    \item[iii.] $S(\tau)$ is continuous in the strong operator topology. Indeed, for $\psi\in S(\mathbb{H}^d)$, it is easy to check that $S(\tau)\psi\to S(\tau_0)\psi $ for $\tau \to \tau_0$ as an application of dominated convergence theorem.
    Since by point 2. the family of operators $(S(\tau))_{\tau\geq 0}$ is uniformly bounded and it converges on a dense set, it converges on $L^2(\mathbb{H}^d)$.
    \item[iv.] By considering again a vector $\psi\in S(\mathbb{H}^d)$ and using representations \eqref{Representation S-tau-Schroedinger case} and \eqref{Representation S-tau-Schroedinger case-2}, we get 
    \begin{align*}
      S(\tau)\psi (z,\alpha) &=  \psi (z,\alpha)-\tau \frac{i}{2}\int_{\bR^{2d}}|\eta +\alpha\tilde z|^2e^{i\eta z}\hat \psi (\eta ,\alpha)\mathrm{d}^{2d}\eta+\tau^2 R(\tau, z,\alpha)\\
      & = \psi (z,\alpha)+\tau \frac{i}{2}L\psi (z,\alpha)+\tau^2  R_\tau( z,\alpha)
    \end{align*}
    where the remainder $  R_\tau( z,\alpha)$ is given by 
    \begin{equation*}
         R_\tau( z,\alpha)=-\frac{1}{4}\int_0^1\int_{\bR^{2d}}(1-u)|\eta{+}\alpha \tilde z|^4e^{i\eta z}e^{-\frac{i}{2}|\eta{+}\alpha \tilde z|^4\tau u}\hat \psi (\eta,\alpha)\mathrm{d}^{2d}\eta \mathrm{d}u\,,
    \end{equation*}
    and we just have to show that, for $\tau$ belonging in a neighborhood of 0, the $L^2$-norm of the map $(z,\alpha)\mapsto   R_\tau( z,\alpha)$ is bounded. By construction, the function $R_\tau( z,\alpha)$ is given by 
    $$R_\tau( z,\alpha)=-\frac{1}{4}\int_0^1\phi_{\tau u}(z,\alpha)(1-u)\mathrm{d}u\,,$$
    where 
    \begin{equation}\label{phiutau}
        \phi_{\tau u}(z,\alpha)=\int_{\bR^{2d}}P(\alpha,\eta,z) e^{i\eta z}e^{-\frac{i}{2}|\eta{+}\alpha \tilde z|^4\tau u}\hat \psi (\eta,\alpha)\mathrm{d}^{2d}\eta\mathrm{d}u\,,
    \end{equation}
    and $P$ is a polynomial function in the variables $\alpha,k,\eta$. In fact the function \eqref{phiutau} is a finite linear combination of terms of the form
    \begin{align}\nonumber
      \phi_{n_\alpha,{\mathbf n},\mathbf{m} }(z,\alpha)& = \alpha^{n_\alpha}\prod_iz_i^{n_i} \int_{\bR^{2d}}\prod_j\eta_j^{m_j} e^{i\eta z}e^{-\frac{i}{2}|\eta{+}\alpha \tilde z|^4\tau u}\hat \psi (\eta,\alpha)\mathrm{d}^{2d}\eta\mathrm{d}u\\
       & =M_{n_\alpha, ,{\mathbf n}}U_2(u\tau)V(u\tau)U_1(u\tau)N_{\mathbf{m}}\psi (\eta,\alpha)\label{rap-pr4sch}
    \end{align}
    where $U_1(u\tau)$, $U_2(u\tau)$ and $V(u\tau) $ are the bounded operators defined by \eqref{U1}, \eqref{U2} and \eqref{V} respectively, while
    \begin{eqnarray*}
         M_{n_\alpha, ,{\mathbf n}}\psi(z,\alpha)&:=&\alpha^{n_\alpha}\prod_iz_i^{n_i}\psi(z,\alpha)\\
    \widehat{N_{\mathbf{m}}\psi}(\eta,\alpha)&:=&\prod_j\eta_j^{m_j}\hat \psi(\eta,\alpha)
    \end{eqnarray*}
       One can easily verify that all the operators appearing on the right hand side of \eqref{rap-pr4sch} map $S(\mathbb{H}^d)$ into itself, hence  the vector $ \phi_{\tau u}$ defined by \eqref{phiutau} belongs to $S(\mathbb{H}^d)$. Moreover, by the boundedness and strong continuity of the operators $U_1(u\tau)$, $U_2(u\tau)$ and $V(u\tau) $ as a function of $\tau$, the map $u\mapsto \|\phi_{\tau u}\|$ is continuous and bounded. Hence, by Jensen's inequality we have:
       \begin{align*}
           \|R_\tau\|^2&=\frac{1}{16}\int_{\bR^{2d}}|\int_0^1\phi_{\tau u}(z,\alpha)(1-u)du|^2dzd\alpha\\
           &\leq \frac{1}{16}\int_{\bR^{2d}}\frac{1}{2}\int_0^1|\phi_{\tau u}(z,\alpha)|^2(1-u)dudzd\alpha\\
           &=\frac{1}{32}\int_0^1\|\phi_{\tau u}\|^2(1-u)du<+\infty\,.
       \end{align*}
      
    \end{enumerate}

    \begin{remark}
        By applying a simple change of variable argument to the oscillatory integral \eqref{Eq: Chernoff approximation for Feynman integral - first definition}, the action of the operator $S(\tau)$ on a vector $\psi\in S(\mathbb{H}^d)$ can be equivalently represented as:
        \begin{align}\label{Eq: Chernoff approximation for  Feynman integral-2}
	[S(\tau)\psi](z,s)
	=\frac{1}{(2\pi i\tau^{-1})^d}\int^o_{\mathbb{R}^{2d}}e^{i\tau|\zeta|^2/2}\psi(z+\tau\zeta,s+\tau z\cdot_\sigma\zeta)\mathrm{d}^{2d}\zeta\,,
\end{align}
    \end{remark}
    \begin{remark}
        A direct consequence of theorem \ref{theorem-Chernoff-approximations-schroedinger-0} is the following approximation formula for the solution of Eq. \eqref{Schrodinger} for $v\equiv 0$:
        \begin{equation}\label{chernoff-schroedinger-0}
    \lim_{n\to\infty}\sup_{\tau\in[0,T]}\left\|S(\tau/n)^n\psi - e^{-iH\tau}\psi\right\|=0
\end{equation}
        \end{remark}
    \subsection{The Feynman-Kac formula }\label{subsectionFeynman-2}
Let us consider now the Schr\"odinger equation \eqref{Schrodinger}
associated to the Hamiltonian operator $H$ given by \eqref{H}. Consider the family of operators $\{S_v(\tau)\}_{\tau\geq 0}$ defined on $L^2(\bH^{ d})$    by
\begin{equation}\label{S-V}
    S_v(\tau)=S(\tau)e^{-iv\tau}
\end{equation}
where, with some abuse of notation, $v:L^2(\bH^{ d})\to L^2(\bH^{ d})$ denotes the multiplication operator associated to the potential $v\in L^\infty(\bH^d)$.

\begin{theorem}\label{theorem-Chernoff-approximations-V} Let $v\in L^\infty(\bH^d)$. Then the family of maps $S_v(\tau))_{\tau\geq 0}$ defined by \eqref{S-V} is Chernoff equivalent to the $C_0$-unitary group $U(\tau)=e^{-iH\tau}$. In particular for all $\psi \in L^2(\bH^d)$ and $T>0$
\begin{equation*}
    \lim_{n\to\infty}\sup_{\tau\in[0,T]}\left\|S_v(\tau/n)^n\psi - e^{-iH\tau}\psi\right\|=0
\end{equation*}
and for $\psi_0\in D(H)$ the map $\tau\mapsto \psi (\tau) =e^{-iH\tau}\psi_0$ is a classical solution of Schr\"odinger equation \eqref{Schrodinger}.
\end{theorem}
\begin{proof}
    By construction, for all $\tau\geq 0$ the operator $S_v(\tau)$ is the composition of $S(\tau)$ and the unitary multiplication operator $M(\tau)=e^{-iv\tau}$.
    Since the family of operators $(M(\tau))_{\tau\in \bR}$ is a $C_0$-group of unitary operators, it automatically fulfills properties i., ii. and iii. of theorem \ref{Chernoff-theorem}. Concerning property iv., by taking a vector $\psi \in S(\mathbb{H}^d)$ we have:
    \begin{align*}
      \frac{S_v(\tau)\psi-\psi+i\tau H\psi}{\tau}&=\frac{S_v(\tau)\psi-\psi-i\tau \frac{L}{2}\psi+i\tau v\psi}{\tau}\\
      &\leq e^{-i\tau v}\left(\frac{S(\tau)\psi-\psi-i\tau\frac{ L}{2}\psi}{\tau}\right)+\frac{e^{-i\tau v}\psi-\psi+i\tau  v\psi}{\tau}+i(e^{-i\tau v}-1)\frac{L}{2}\psi\,.
    \end{align*}
    The first term converges  to 0 since the operator $M(\tau)=e^{-i\tau v}$ is bounded and $S(\tau)$ satisfies condition \ref{Item: Chernoff conditions - approximation on a core}  of Theorem \ref{Chernoff-theorem} (see the proof of Theorem \ref{theorem-Chernoff-approximations-schroedinger-0}), the second term converges to 0 since $S(\mathbb{H}^d)\subset D(v)$, while the last term converges to 0 by the strong continuity of the unitary group $(M(\tau))_{\tau\in \bR}$ generated by the multiplication operator $v$.
\end{proof}
\begin{remark}
    The same conclusion of Theorem \ref{theorem-Chernoff-approximations-V} holds by switching the order of operators $S(\tau)$ and $T(\tau)$, obtaining:
    \begin{equation}\label{chernoff-schroedinger-V}
    \lim_{n\to\infty}\sup_{\tau\in[0,T]}\left\|(M(\tau/n)S(\tau/n))^n\psi - e^{-iH\tau}\psi\right\|=0
\end{equation}
\end{remark}
    \subsection{Interpretation in terms of Feynman path integrals}
Formulae \eqref{chernoff-schroedinger-0} and \eqref{chernoff-schroedinger-V} have a direct interpretation in terms of the so-called Feynman path integral representation, a heuristic construction procedure for the solution of the Schr\"odinger equation which is extensively applied in the physics literature. Feynman formulae were introduced in the early 40's in R. Feynman's PhD thesis\cite{FeyThesis,Fey}, where he proposed a suggestive representation for the solution of the Schr\"odinger equation describing the time evolution of the state $\psi $ of a non relativistic quantum particle, the so-called wave function:
\begin{equation}\label{Schroedinger traditional}
    \begin{cases}
        i\frac{\partial}{\partial t}\psi(t,x)=-\frac{1}{2}\Delta\psi(t,x) +v(x)\psi(t,x),\qquad x\in \bR^d,\,  t\in \bR\,.\\
         \psi(0,x)=\psi_0(x)
    \end{cases}
\end{equation}
According to Feynman's proposal, the solution $\psi $ should be given by an integral formula of the following form

    \begin{equation}\label{Fey1}
\psi (t,x)=C^{-1}\int_\Gamma e^{iS(\gamma)}\psi_0(0, \gamma (t))d\gamma\,,
\end{equation}
where $\Gamma$ denotes the set of continuous paths $\gamma :[0,t]\to \bR^d$ with fixed endpoint $\gamma(0)=x$, the symbol $d\gamma$ stands for a Lebesgue-type measure on $\Gamma $, $C$ is a normalization constant while the integrand 
$$
S(\gamma)=\int{\mathcal L}(\gamma(\tau),\dot\gamma(\tau) )d\tau=\int_0^t\left( \frac{|\dot\gamma(\tau)|^2}{2}-v(\gamma(\tau))\right)d\tau,$$ is the classical action functional  evaluated along the path $\gamma$, where $\mathcal L$ denotes the Lagrangian.

Clearly, formula \eqref{Fey1} lacks of mathematical rigour. Indeed, neither the Lebesgue measure on the infinite-dimensional space $\Gamma$ nor the normalization constant $C$ are well defined and several efforts have been devoted to the rigorous mathematical definition of formula \eqref{Fey1} (see, e.g. \cite{JoLa,Mabook,Fu,NiTrabook} and references therein).  The minimal interpretation of representation \eqref{Fey1} is in terms of the limit of a particular approximation procedure for the solution of \eqref{Schroedinger traditional}.
Indeed, by dividing the time interval $[0,t]$ into $n$ equal parts of amplitude $t/n$ and  considering the finite-dimensional space $\Gamma_n$ of piecewise linear paths $\gamma_n$, i.e. those attaining constant velocities on the partition subintervals 
\begin{equation}\label{Piecewise linear paths}
	\gamma_n(\tau):=x_j+\frac{(x_{j+1}-x_j)}{t/n}(\tau -jt/n), \ \ \tau\in \left [\frac{jt}{n},\frac{(j+1)t}{n} \right],
\end{equation}
where $x_j:=\gamma (jt/n)$ and $j=0, \dots, n-1$, then formula \eqref{Fey1} reduces to a finite dimensional oscillatory integral over the set of positions $x_j$ attained by the paths $\gamma_n $ at the endpoints $t_j$ of the partition subintervals
\begin{equation}
\label{trotter2}
(2 \pi i  t/n)^{-nd/2}\int^o_{\bR^{nd}} \psi_0(x_n) e^{i\sum_{j=0}^{n-1}\frac{|x_{j+1}-x_{j}|^2}{2t/n}-i \sum_{j=1}^nv(x_j)\frac{t}{ n}}dx_1\dots dx_{n}.
\end{equation}
According to Feynman's intuition, the sequence \eqref{trotter2} should converge for $n\to \infty$ to the solution of Eq. \eqref{Schroedinger traditional}, and this Ansatz has been proved in various topologies and under different assumptions on the potential $v$ (see, e.g., \cite{Nel, Nic16,JoLa,Tru1,Ichi0}). In the case where the Euclidean configuration space $\bR^d$ is replaced by a general Riemannian manifold, the set $\Gamma_n$ of piecewise linear paths, on which the approximation procedure \eqref{trotter2} relies, is replaced by a set of piecewise geodesic paths. Contrary to on the Euclidean case, where a rich mathematical literature is present, for Riemannian manifolds only a few rigorous results have been obtained, mostly restricted to rather peculiar compact manifolds with additional algebraic structure such as compact Lie groups or simmetric spaces \cite{Tha1,Tha2,DrMaMo23}, see also the   recent work \cite{Fuk} where approximations via piecewice classical paths are considered.
In this section we show how the Chernoff approximations developed in Sections \ref{subsection-Fey1} and \ref{subsectionFeynman-2} provide a rigorous mathematical proof of Feynman construction procedure via piecewise geodesic paths on the Weyl-Heisenberg group. In the following we shall restrict ourselves to the case $v\in C^\infty_b(\bH^d)$.

First of all, let us remark that by applying an elementary change of variable in the regularized integrals appearing on the right-hand-side of \eqref{limit definition oscillatory integral 1}, for any $\tau>0$ the action of the operator $S(\tau)$ on $S(\mathbb{H}^d)$ can be equivalently written as
\begin{align}\label{Eq: Chernoff approximation for  Feynman integral}
	[S(\tau)\psi](z,s)
	=\frac{1}{(2\pi i\tau^{-1})^d}\int^o_{\mathbb{R}^{2d}}e^{i\tau|\xi|^2/2}\psi(z+\tau \xi ,s+\tau z\cdot_\sigma \xi)\mathrm{d}^{2d}\xi\,,
\end{align}
where the new variable $\xi\equiv (\xi^{x,1},\ldots,\xi^{x,d}, \xi^{y,1},\ldots,\xi^{y,d})\in \bR^{2d}$ provides the components, under the basis $\{X_1,\ldots, X_d, Y_1,\ldots, Y_d\}$ of the horizontal subspace of the Lie algebra, of the velocity vector of the paths $\tau\mapsto (z+\tau \xi ,s+\tau z\cdot_\sigma \xi) $. 
By introducing this notation and relying on the fact that $S(\tau)$ maps $S(\mathbb{H}^d)$ into itself  and Eq. \eqref{chernoff-schroedinger-0}, for all $t>0$ it it possible to approximate the solution $\psi (t)$ of the Schr\"odinger equation \eqref{Schrodinger} with $v=0$ and $\psi_0\in S(\mathbb{H}^d)$ as the limit in $L^2(\mathbb{H}^d)$ of the sequence of vectors $\psi_n(t)\in S(\mathbb{H}^d)$ given by
\begin{equation}\label{fin dim schrodinger V0}
    \psi_n(t)( z,s)=(2\pi i(t/n)^{-1})^{-nd}\int^o_{\mathbb{R}^{2d}}\dots\int^o_{\mathbb{R}^{2d}}\psi_0(\gamma_{\xi_1,\dots,\xi_n}(t))e^{\frac{i}{2}\sum_{j=1}^n|\xi|^2\frac{t}{n}}\mathrm{d}^{2d}\xi_1\dots \mathrm{d}^{2d}\xi_n
\end{equation}
where $\gamma_{\xi_1,\dots,\xi_n}:[0,t]\to \bH^d$ is the piecewise-horizontal curve with components $(z(r),s(r))$ defined by
\begin{equation}\label{piecewise geodesic WH}
\begin{cases}
    \gamma_{\xi_1,\dots,\xi_n}(0)=(z,s)\\
    \gamma_{\xi_1,\dots,\xi_n}(r)=(z(t_j)+(r-t_j)\xi_{j+1},s(t_j)+(r-t_j)z(t_j)\cdot_\sigma v_{j+1}), \quad r\in \left[t_j,t_{j+1}\right]
\end{cases}
\end{equation}
where $t_j:=jt/n$, $j=0,\dots, n$.
In particular, $\gamma_{\xi_1,\dots,\xi_n}$ is a piecewise-geodesic path, since for each subinterval $\left[jt/n,(j+1)t/n\right]$ the curve $\gamma_{\xi_1,\dots,\xi_n}$ is a geodesic with velocity $\sum_{k=1}^d\left(\xi_{j+1}^{x,k}X_k+\xi_{j+1}^{y,k}Y_k\right)$.

More generally, in the case where  $v\in C^\infty_b(\mathbb{H}^d)$ we have that the unitary group $M(\tau)=e^{-i\tau v}$, $\tau\in \bR$ maps $S(\mathbb{H}^d)$ into itself and formula  \eqref{chernoff-schroedinger-V} for $\psi_0\in S(\mathbb{H}^d)$ can be interpreted as the convergence in $L^2(\mathbb{H}^d)$ of the sequence of vectors $\psi_n(t)\in S(\mathbb{H}^d)$ given by:
\begin{multline}\label{Feynman-approximations}
    \psi_n(t)( z,s)=(2\pi i(t/n)^{-1})^{-nd}\int^o_{\mathbb{R}^{2d}}\dots\int^o_{\mathbb{R}^{2d}}\psi_0(\gamma_{\xi_1,\dots,\xi_n}(t))e^{i \left(\frac{1}{2}\sum_{j=1}^n|\xi|^2\frac{t}{n}-\sum_{j=1}^nv(\gamma_{\xi_1,\dots,\xi_n}(jt/n))\frac{t}{n}\right)}\\ \mathrm{d}^{2d}\xi_1\dots \mathrm{d}^{2d}\xi_n
\end{multline}
to the solution $\psi(t)\equiv e^{-it H}\psi_0$ of the Schr\"odinger equation \eqref{Schrodinger}. The integral is meant as an iterative oscillatory integral in the sense of Eq. \eqref{limit definition oscillatory integral 1}. The exponent $$\left(\frac{1}{2}\sum_{j=1}^n|\xi|^2\frac{t}{n}-\sum_{j=1}^nv(\gamma_{\xi_1,\dots,\xi_n}(jt/n))\frac{t}{n}\right) $$ appearing on the right-hand-side of formula \eqref{Feynman-approximations} can be interpreted as the classical action functional evaluated along the path $\gamma_{\xi_1,\dots,\xi_n}$. In particular, the term $\frac{1}{2}\sum_{j=1}^n|\xi|^2\frac{t}{n}$ is associated to the  kinetic part of the classical action $\frac{1}{2}\int_0^\tau\|\dot \gamma_{\xi_1,\dots,\xi_n} (r)\|^2dr$ by computing the norm of the vector $\dot \gamma_{\xi_1,\dots,\xi_n} $ with respect to an  horizontal metric on $\bH^d$ making the vectors $X_j$, $Y_j$ in \eqref{basis vectors} orthonormal, while the term $\sum_{j=1}^nv(\gamma_{\xi_1,\dots,\xi_n}(jt/n))\frac{t}{n}$ provides the Riemann sum approximation of  the integral $\int_0^t v(\gamma_{\xi_1,\dots,\xi_n}(r))dr$. 

This analysis shows how, in a sub-Riemannian manifold, the paths $\gamma\in \Gamma$ on which the heuristic integration in  \eqref{Fey1} is performed are to be chosen among the horizontal ones.
\section{Relation with the magnetic Laplacian}

For any vector $\psi\in L^2(\bH^d)$, let us consider its representation in terms of the partial Fourier transform $\hat \psi$ defined as
\begin{equation}\label{patial fourier transform}
    \tilde\psi (z, \alpha):=\frac{1}{2\pi}\int_\bR e^{-i\alpha s}\psi (z, s) \dr s\,.
\end{equation}

In this representation, the horizontal vectors \eqref{basis vectors} of the Lie algebra assume the following form:
\begin{equation*}
     X_j=\partial_{x^j}+i\alpha y^j\,,
    \qquad
    Y_j=\partial_{y^j}-i\alpha x^j\,,
\end{equation*}
and the horizontal Laplacian $L$ on smooth compactly supported functions $\tilde \psi \in C^\infty_c(\mathbb{H}^d)$ is given by:
\begin{align*}
    L\tilde \psi (z, \alpha)&=\sum_{j=1}^d(X_i^2+Y_i^2)\tilde \psi (z, \alpha)\\
    &=-\sum_{j=1}^d\left( (-i\partial_{x^j}+\alpha y^j )^2+(-i\partial_{y^j}-\alpha x^j)^2\right)\tilde \psi (z, \alpha)\,.
\end{align*}

The sum above can be written as $(-i\nabla -A)^2$, where $A$ denotes the vector field on $\bR^{2}$ defined as 
\begin{equation}\label{vectorA}
A(x^j,y^j)=(-\alpha y^j, \alpha x^j) \,,    
\end{equation}and represents the so-called magnetic Laplacian associated to a constant magnetic field $B=\nabla \times A=2\alpha$. 

In this representation, the heat equation \eqref{CauchyProblemHeat} and the Schr\"odinger equation \eqref{Schrodinger}, with $d=1$ and respectively $c=0$ and $v=0$, assume the following form:
\begin{equation}\label{heat-magnetic}
     \partial_t \psi_\alpha (t, x,y)=-\frac{1}{2}(-i\nabla -A)^2 \psi_\alpha(t,x,y)
 \end{equation}
 \begin{equation}\label{schrodinger-magnetic}
     i\partial_t \psi_\alpha (t, x,y)=\frac{1}{2}(-i\nabla -A)^2\psi_\alpha(t,x,y)
 \end{equation}
 where  $\psi_\alpha (t,x,y)\equiv \tilde \psi (t, x,y, \alpha)$. In particular, Eq. \eqref{schrodinger-magnetic}  describes the time evolution of the state of  a charged quantum particle moving in the 2-dimensional Euclidean space under the action of a constant magnetic field.

This is an exactly solvable model, the solution being given by $$\psi_\alpha (t, x,y)=\int_{\bR^2}K_\alpha(x,y, x',y')\psi_\alpha (0, x',y')dx'dy'$$ with:
 \begin{equation*}
     K_\alpha(x,y, x',y')=\frac{\alpha}{2\pi \sinh (\alpha t)}\frac{1}{\sqrt{2\pi t}}e^{-\frac{\alpha }{2}\coth (\alpha t)\left((x-x')^2+((y-y')^2\right)-i\alpha (xy'-x'y)}
 \end{equation*}
 in the case of the heat equation \eqref{heat-magnetic} and
\begin{equation*}
     K_\alpha (x,y, x',y')=\frac{\alpha}{2\pi \sin (\alpha t)}\frac{1}{\sqrt{2\pi i t}}e^{-\frac{\alpha }{2}\cot (\alpha t)\left((x-x')^2+(y-y'^2\right)-i\alpha (xy'-x'y)}
 \end{equation*}
  in the case of the Schr\"odinger equation \eqref{schrodinger-magnetic} (see, e.g. \cite{Simon} and references therein).
 Moreover, the functional integral representation for the solution of Eq. \eqref{heat-magnetic} and Eq. \eqref{schrodinger-magnetic}, i.e. the Feynman-Kac-Ito formula \cite{Simon} and the Feynman path integral \cite{AlBrMF,AlCaMa}, have been extensively studied.
 \subsection{The Feynman-Kac-Ito formula}
 The classical Feynman-Kac formula for the probabilistic representation of the solution of the heat equation on $\bR^d $ in terms of the expectation with respect to the distribution of Brownian  motion can be extended to the case of equation \eqref{heat-magnetic} and assumes the following form:
 \begin{equation}\label{FK-magnetic-1}
     \psi_\alpha (t,x,y)=\bE\left[e^{i\alpha\left(yB_1(t)-xB_2(t)+\int_0^t B_2(r)dB_1(r)-B_1(r)dB_2(r))\right)}\psi_\alpha (0,x+B_1(t),y+B_2(t)) \right]\,,
 \end{equation}
 Where $B_1,B_2$ are two independent standard Brownian motions and $dB$ denotes the Ito stochastic integral. The term ${\mathcal A}(t):=\int_0^t B_2(r)dB_1(r)-B_1(r)dB_2(r))$ represents the {\em Levy stochastic area}.
 The solution $\psi_t(z,s) $ can be recovered via inverse Fourier transform $$\psi_t(z,s)=\int_\bR e^{i\alpha s}\tilde \psi (z,\alpha)\dr \alpha$$
 yielding
 \begin{align}\label{Feynman-Kac-magnetic-heat}
     \psi_t(z,s)&=\int_\bR e^{i\alpha s}\bE\left[e^{i\alpha\left(yB_1(t)-xB_2(t)+\int_0^t B_2(r)dB_1(r)-B_1(r)dB_2(r))\right)}\hat\psi_0 (x+B_1(t),y+B_2(t),\alpha) \right]\dr\alpha\\
     &=\bE\left[\psi_0 \left(x+B_1(t),y+B_2(t),yB_1(t)-xB_2(t)+\int_0^t B_2(r)dB_1(r)-B_1(r)dB_2(r)\right) \right]
 \end{align}
 \begin{remark}
 The formula for general magnetic fields contains in the exponent an additional term of the form $\frac{1}{2}\int_0^t\Div A(x+B_1(r),y+ B_2(r))\dr r$, which in our case vanishes since the vector potential $A$ in \eqref{vectorA} is divergence-free. The correction term is exactly the difference between Ito stochastic integral $\int_0^tA(B(r))\dr B(r)$ and Stratonovich stochastic integral $\int_0^tA(B(r))\circ \dr B(r)$ and the general form of formula \eqref{FK-magnetic-1} is
 \begin{equation*}
     \psi_\alpha (t,x,y)=\bE\left[ e^{-i\int_0^tA(x+B_1(r),y+B_2(r))\circ \dr (B_1(r),B_2(r))}\psi_\alpha (0,x+B_1(t),y+B_2(t)) \right]\,.
 \end{equation*}

 \end{remark}
\subsection{Feynman path integral construction in terms of infinite dimensional oscillatory integrals}
In the realm of the different approaches to the mathematical definition of the heuristic Feynman formula \eqref{Fey1}, several efforts have been devoted to overcome the "minimal" one, relying essentially on the proof of the convergence of the sequence \eqref{trotter2} to the solution of the Schr\"odinger equation \eqref{Schroedinger traditional}. However, the presence of the imaginary unit $i$ in the heuristic Feynman formula \eqref{Fey1} and in its finite-dimensional approximations \eqref{trotter2} does not allow to construct the Feynman integral in terms of a well defined (in Lebesgue sense) integral over a space $\Gamma$ of continuous paths with respect to a complex $\sigma$-additive measure on it. In other words, for Schr\"odinger equation it is not possible to construct a measure on $\Gamma$ that plays the same role played by the Wiener measure for the heat equation (see \cite{Cam,Tho,Mabook} for a discussion of this no-go result). Due to these difficulties, it becomes necessary to adopt a different point of view to infinite dimensional integration and realize the "Feynman integral" as a linear continuous functional on a suitable class of integrable functions, in the spirit of Daniell's approach to integration theory. This line of thought lead to different approaches to Feynman integration, including e.g. infinite dimensional distribution theory and white noise analysis \cite{HKPS}, analytic continuation of Wiener integrals \cite{JoLa}, infinite dimensional oscillatory integrals.   In particular, the latter technique allowed to address non trivial quantum models \cite{AlGuMa,AlMa2} and to study, via an infinite dimensional version of the classical stationary phase method\cite{Hor,Dui}, the detailed semiclassical asymptotic behavior of the solution of the Schr\"odinger equation in the limit  where the (reduced) Planck $\hbar$ is regarded as a small parameter which is allowed to converge to 0. This approach is based on the generalization of the definition and the main properties of classical oscillatory integrals  \cite{Hor} to the case where the integration domain is a real separable infinite dimensional Hilbert space \cite{AlBr,ELT,Mabook} and will be adopted in the rest of this section.

Finite dimensional oscillatory  integrals are objects of the following form 
\begin{equation}\label{osc1}
\int_{\bR^n}e^{i\Phi(x)}f(x)\dr x,
\end{equation}
where  $\Phi:\bR^n\to \bR$ and $f:\bR^n\to \bC$ are Borel functions, $\Phi$ is usually called {\em phase function}. 
Particular example of integrals of this form are the so-called {\em Fresnel integrals}, where  $\Phi$ is a non-degenerate quadratic form.
They find applications in optics and in the theory of wave diffraction and, from a purely mathematical point of view, they have been extensively studied in connection with the theory of Fourier integral operators \cite{Hor, Dui}. Remarkably, even in the cases where the function $f$ is not summable, the oscillatory  integrals \eqref{osc1} can be defined and computed  as the limit of a sequence of regularized integrals. More specifically, as mentioned in Eq. \eqref{limit definition oscillatory integral 1}, if for any test function $\phi\in S(\bR^n)$ such that $\phi(0)=1$ and for all $\epsilon>0$ the integral $I_{\epsilon,\phi}\equiv\int_{\bR^n}\phi(\epsilon x )e^{i\Phi(x)}f(x)\dr x$ is well defined (in Lebesgue sense) and if it exists the limit $\lim_{\epsilon\downarrow 0}I_{\epsilon,\phi}\equiv I$ and it is independent of $\phi$, then it is called  {\em oscillatory integral of $f$ with respect to the phase $\Phi$ } and denoted $\int^o_{\bR^n}e^{i\Phi(x)}f(x)\dr x$. This definition allows to take advantage of the cancellations due to the oscillatory behaviour of the integrand and to prove, e.g., identities such as $\int^o_{\bR^n} e^{i\frac{\|x\|^2}{2}}\dr x=(2\pi i)^{n/2}$. In the case the phase $\Phi$ is a non-degenerate quadratic form, the oscillatory integral is called {\em Fresnel integral}.

In \cite{AlHK,ELT,AlBr} the definition of Fresnel integral has been extended to the case where the underlying integration space $\bR^n$ is replaced by an infinite dimensional real separable Hilbert space $\mathcal H$. The definition of an {\em infinite dimensional Fresnel integral}, formally denoted $ \widetilde{\int_\Hi } e
^{i\frac{\| x\|^2}{2} }f(x)\dr x$ relies on an approximation procedure via finite dimensional integrals which allows to overcome the non existence of the Lebesgue measure $\dr x$ on  infinite dimensional Hilbert spaces \cite{Yam,Mabook}.

\begin{definition}
\label{intoscinf1}
A function $f:\Hi\to\bC$ is  said to be {\em Fresnel integrable}
if  for any sequence $\{P _n\}_n$ of projectors
onto n-dimensional subspaces of $\Hi$, such that $P_n \to \mathbb{I}$ strongly as $ n \to \infty$ ( $\mathbb{I}$ being the identity operator in $\Hi$), the  oscillatory integrals
$$  
\int _{P _n\Hi}^oe ^{i\frac{\|P _n x\|^2}{2} }f(P _n x )d (P _nx ),
$$
are well defined  and the limit
\begin{equation}
\label{infdim}
\lim_{ n \to \infty} (2\pi i  )^{-n/2}  \int_{P _n\Hi}^oe ^{i\frac{\|P _n x\|^2}{2} }f(P _n x )d (P _nx )
\end{equation}
exists and is independent of the sequence $\{ P _n\}_n$. In this case the limit is  called {\em infinite dimensional oscillatory (Fresnel) integral} of
$f$  and is denoted by $$ \widetilde{\int_\Hi } e
^{i\frac{\| x\|^2}{2} }f(x)dx.$$
\end{definition}
\begin{remark}\label{remark independence of Pn}
    It is worthwhile to point out the importance of the requirement that the limit of the sequence of finite dimensional approximations \eqref{infdim} does not depend on the choice of the sequence of finite dimensional projectors $\{P_n\}_n$. This makes the Fresnel integral $\widetilde{\int_\Hi } e
^{i\frac{\| x\|^2}{2} }f(x)dx$ a well-defined object, whose value does not depend on the specific finite dimensional approximation implemented for its computation. 
\end{remark}

The description of the full class of Fresnel integral functions is still an open problem of harmonic analysis, even in finite dimension. However, it is possible to provide some explicit examples of subclasses $\cF\subset \bC^\Hi$ of Fresnel integrable function where explicit representation formulae for the Fresnel integral can be proved. In addition, under suitable topologies the infinite dimensional Fresnel integral turns out to be a linear continuous functional on $\cF$ \cite{AlHK,AlMa16,MaNiTr}.

Within this abstract framework, the heuristic Feynman path integral representation \eqref{Fey1} for the solution of the Schr\"odinger equation \eqref{Schroedinger traditional} can be rigorously mathematically realized as an infinite dimensional oscillatory integral on a suitable Hilbert space of continuous ``paths''. Let us consider the so-called {\em Cameron-Martin space} $\Hi_t$, that is the Hilbert space of absolutely continuous paths $\gamma :[0,t]\to\bR^d$ with fixed initial point $\gamma(0)=0$  and square integrable weak derivative $\int_0^t|\dot\gamma(r)|^2\dr r<\infty$, endowed with the  inner product
$$\langle \gamma _1,\gamma _2\rangle =\int ^t _0 \dot\gamma _1(r)\cdot\dot\gamma
  _2(r)\dr r .$$
  Under suitable assumptions on the potential $v$ and the initial datum $\psi_0$ it is possible to prove that  
   the function $f:\Hi\to \bC$ given by 
   $$f(\gamma)=e^{-i\int_0^t v(\gamma (r)+x)\dr r}\psi _0(\gamma (0)+x), \qquad x\in \bR^d, \gamma \in \Hi_t,$$
   is Fresnel integrable. Further, its infinite dimensional oscillatory integral
   $$\widetilde{\int_{\Hi_t}} e^{\frac{i}{2}\langle \gamma, \gamma \rangle}f(\gamma)d\gamma\equiv \widetilde{\int_{\Hi_t}} e^{\frac{i}{2}\int_0^t |\dot\gamma(r)|^2\dr r}e^{-i\int_0^t v(\gamma (r)+x)\dr r}\psi _0(\gamma (0)+x)d\gamma,$$
provides a representation for the solution  of \eqref{Schroedinger traditional} (see, e.g., \cite{AlHK,ELT,AlBr,Mabook}). The traditional "minimal" construction of the Feynman integral in terms of the limit of the approximations on piecewise-linear paths (see Eq. \eqref{trotter2}) can be recovered in the framework of infinite dimensional oscillatory integrals by considering a specific sequence of finite dimensional projection operators $\{P_n\}_n$.  Indeed, for any $n\in \bN$ consider the projector operator  $P_n:\Hi_t\to\Hi_t$ onto the subspace $\Gamma_n\subset \Hi_t$ of paths of the form \eqref{Piecewise linear paths}. In this case the finite dimensional approximations appearing in formula \eqref{infdim} resemble the integrals appearing in the traditional construction \eqref{trotter2}.

In \cite{AlCaMa} the Feynman path integral  for the Schr\"odinger equation associated to the magnetic Laplacian \eqref{schrodinger-magnetic}, i.e. the heuristic representation formula 
\begin{equation*}
       \psi (t,x)=\int _{\gamma(0)=x}e^{\frac{i}{2}\int_0^t|\dot\gamma(r)|^2\dr r{-}i\int_0^t  A(\gamma(r))\cdot \dot \gamma (r)\dr r }
      \psi_0(\gamma(t))d\gamma,
  \end{equation*}
has been studied within the framework of infinite dimensional oscillatory integrals. In particular, for a linear vector potential of the form \eqref{vectorA} two interesting results have been obtained. First, by taking an initial datum $\psi_0\in S(\mathbb{H}^d)$ such that its Fourier transform $\hat \psi_0$ is compactly supported and considering the sequence $\{P_n\}_n$ of projection operators onto the subspaces $\Gamma_n$ of piecewise linear paths, it is possible to show that the sequence of finite dimensional approximations of the infinite dimensional oscillatory integral over the Cameron Martin space
\begin{equation}
\label{piecewise linear magnetic}
       (2 \pi i   t/n)^{-nd/2}\int^o_{\bR^{nd}} \psi_0(x_n) e^{\frac{i}{2}\sum_{j=0}^{n-1}\frac{|x_{j+1}-x_{j}|^2}{t/n}{-}i\sum_{j=0}^{n-1}\int_{jt/n}^{(j+1)t/n}A\left(x_j+r\frac{(x_{j+1}-x_i)}{t/n}\right)\cdot (x_{j+1}-x_i)\dr r}dx_1\dots dx_{n}.,
  \end{equation}
with $x_0\equiv x$, converges to the solution $\psi (t, x)$ of Eq. \eqref{schrodinger-magnetic}. Actually, this result holds for a large class of vector potentials $A$, including also the linear ones. In particular for $d=1$ and $A$ of the form \eqref{vectorA}, by adopting the notation $z_j=(x_j,y_j)\in \bR^2$ $j=1,\ldots , n$ the action integrals appearing in the exponent in \eqref{piecewise linear magnetic} reduces to
\begin{equation*}
    \int_{jt/n}^{(j+1)t/n}A\left(z_j+r\frac{(z_{j+1}-z_i)}{t/n}\right)\cdot (z_{j+1}-z_i)\dr r=\alpha \left(x_j(y_{j+1}-y_j)-y_j(x_{j+1}-x_j)\right)\,.
\end{equation*}

Let us consider again the Schr\"odinger equation \eqref{Schrodinger} on $\bH$ with $v=0$ and let us assume that the initial datum $\psi_0\in S(\bR^{3})$ has a compactly supported Fourier transform $\hat \psi_0(\eta,\alpha)=\frac{1}{(2\pi)^{3}}\int_{\bR^{3}}e^{-i(\eta \cdot z+\alpha s)}\psi_0 (z,s)\mathrm{d}^{2d}z\mathrm{d}s$. In particular, this implies that for each $\alpha\in \bR$ the vector $\psi_\alpha(0)\in S(\bR^3)$ defined by $\psi_\alpha(0)(x,y)=\tilde \psi_0(x,y,\alpha) $ has a compactly supported Fourier transform and therefore the results of \cite{AlCaMa} can be applied. In particular,
 the finite dimensional approximations \eqref{piecewise linear magnetic} for the solution $\psi_\alpha(t, z)$ of \eqref{schrodinger-magnetic} are equal to
\begin{equation*}
       (2 \pi i t/n)^{-n}\int^o_{\bR^{2n}} \psi_0(z_n) e^{i\sum_{j=0}^{n-1}\frac{|z_{j+1}-z_{j}|^2}{2t/n}{-}i\alpha\sum_{j=0}^{n-1}\left(x_j(y_{j+1}-y_j)-y_j(x_{j+1}-x_j)\right)}dz_1\dots dz_{n}\,,
  \end{equation*}
with $z_0\equiv z$.  By introducing the variables $\xi_j:=\frac{z_j-z_{j-1}}{t/n}$, $j=1,\ldots,n$, resorting to the partial inverse Fourier transform of \eqref{patial fourier transform}
\[ \psi (z, s)=\int_\bR e^{i\alpha s}\hat\psi (z, \alpha) \dr s\,\]
and exploiting the continuity of the Fourier transform in the $L^2$-topology, we get the following sequence of approximations for the solution of Eq. \eqref{Schrodinger} on $\bH^d$ in the case $d=1$ and $v=0$:
    \begin{align*}
        \psi_n(t)(z,s)&=(2\pi i(t/n)^{-1})^{-n}\int _\bR \int^o_{\bR^{2n}}\psi_0(z_n,\alpha)e^{i\alpha (s+ \sum_{j=0}^{n-1}z_j\cdot_\sigma \xi_{j+1})} e^{\frac{i}{2}\sum_{j=1}^n|\xi_j|^2\frac{t}{n}}\mathrm{d}^{2d}\xi_1\dots \mathrm{d}^{2d}\xi_n\dr s\\
        &=(2\pi i(t/n)^{-1})^{-n} \int^o_{\bR^{2n}}\psi_0\left(z_n,s+\sum_{j=0}^{n-1}z_j\cdot_\sigma \xi_{j+1}\right) e^{\frac{i}{2}\sum_{j=1}^n|\xi_j|^2\frac{t}{n}}\mathrm{d}^{2d}\xi_1\dots \mathrm{d}^{2d}\xi_n
    \end{align*}
where $z_j=z_0+\sum_{k=1}^j\xi_kt/n$, $j=1,\ldots, n$ and $z_0\equiv z$. Actually the last line of the equation above resembles exactly the form of the finite dimensional approximations \eqref{fin dim schrodinger V0}. Indeed, the point $(z_n,s+\sum_{j=0}^{n-1}z_j\cdot_\sigma \xi_{j+1}) $ coincides with the position $ \gamma_{\xi_1,\dots,\xi_n}(t)$ attained at time $t$ by the piecewise geodesic path $\gamma_{\xi_1,\dots,\xi_n}$ defined by \eqref{piecewise geodesic WH}.

As already pointed out in Remark \ref{remark independence of Pn}, a crucial feature of definition \ref{intoscinf1} is the requirement that the limit of the sequence of finite dimensional approximations \eqref{infdim} does not depend on the sequence of finite dimensional projection operators $\{P_n\}$.
However, as extensively discussed in \cite{AlCaMa}, in the presence of a non trivial magnetic field the limit of the finite dimensional approximations of the  oscillatory integral 
\begin{equation}\label{in-osc-inf-dim-magnetic}
    \int _{\Hi_t}e^{\frac{i}{2}\int_0^t|\dot\gamma(r)|^2\dr r-i\int_0^t  A(x+\gamma(r))\cdot \dot \gamma (r)\dr r }
      \psi_0(x+\gamma(t))d\gamma
\end{equation}
do explicitly depend on the sequence of projections $\{P_n\}$ on the Cameron-Martin space $\Hi_t$. However, it is possible to recover the uniqueness of the limit by introducing in the finite dimensional approximations of \eqref{in-osc-inf-dim-magnetic} some specific renormalization couterterms. 

More precisely, considered a generic orthonormal basis $\{e_n\}$ of $\Hi_t$ and the corresponding sequence of projectors $P_n:\Hi_t\to \Hi_t$ onto the span of the first $n$ vectors $\{e_1,\dots, e_n\}$, it is possible to prove that the sequence of renormalized finite dimensional oscillatory integrals
\[\left( 2\pi i  \right )^{-n/2} \int_{P_n{\Hi_t}}^o e^{\frac{i}{2} \|\gamma_n \|^2 }e^{-i\int_0^tA(\gamma_n(r)+x)\dot\gamma_n(r) \dr r-R_n}\psi_0(\gamma_n(t)+x)d \gamma _n\]
converges to the solution of Eq. \eqref{schrodinger-magnetic} independently of the choice of the basis $\{e_n\}$.  The additional renormalization terms $\{R_n\}$ are explicitly given by 
\begin{equation}\label{renormalization}
    R_n=\frac{B}{2}\sum_{k=1}^n \int_0^t e_k (r)\wedge \dot{e}_k (r) \dr r\,,
\end{equation}
where  $B=\rot  A$ and each term in the sum on the right-hand side of \eqref{renormalization} provides the area integral swept by the path $r\mapsto e_k(r)$ in $\bR^2$. In the particular case where the vector potential $A$  is given by \eqref{vectorA}, $R_n$ reduces to
\[R_n=\alpha \sum_{k=1}^n \int_0^t e_k (r)\wedge \dot{e}_k (r) \dr r=\sum_{k=1}^n\int_0^tA(e_k(r))\cdot \dot e_k(r) \dr r\,,\]
and by resorting to the partial inverse Fourier transform of \eqref{patial fourier transform} and to the continuity of the Fourier transform in the $L^2$ topology, it is possible to obtain the  approximating sequence for the solution of Eq. \eqref{Schrodinger} (with $v=0$ and $d=2$):
\begin{equation}\label{general approximations magnetic field}
    \int_{P_n{\Hi_t}}^o \frac{ e^{\frac{i}{2} \|\gamma_n \|^2 } }{( 2\pi i   )^{n/2}} \psi_0\Big(z+\gamma_n(t), s{-}\int_0^t(z+\gamma_n(r))\wedge \dot \gamma_n(r)\dr r +i \sum_{k=1}^n \int_0^t e_k (r)\wedge \dot{e}_k (r) \dr r\Big)d \gamma _n\,,
\end{equation}
or, equvalently
      \begin{equation}\label{general approximations magnetic field-2}
          \int_{P_n{\Hi_t}}^o \frac{ e^{\frac{i}{2} \|\gamma_n \|^2 } }{( 2\pi i   )^{n/2}} \psi_0\Big(z+\gamma_n(t), s{+}\int_0^t(z+\gamma_n(r))\cdot_\sigma  \dot \gamma_n(r)\dr r -i \sum_{k=1}^n \int_0^t e_k (r)\cdot_\sigma  \dot{e}_k (r) \dr r\Big)d \gamma _n\,.
      \end{equation}
 It is important to note that formulae \eqref{general approximations magnetic field} and \eqref{general approximations magnetic field-2} are meaningful thanks to the assumption on the compactness of the support of the Fourier transform of the initial datum $\psi_0$. Indeed, in this case the function $\psi_0$ can be extended to a holomorphic function on $\bC^3 $ and the integrand in \eqref{general approximations magnetic field} is well defined.

   In the particular case of projection operators onto the subspaces $\Gamma_n$ of piecewise linear paths, it can be easily verified that $R_n=0$ (since the path $\{e_n\}_n$ spanning $\Gamma_n$ swept a vanishing area, while the term $z\cdot_\sigma  \gamma_n(t)-\int_0^t\gamma_n(r)\wedge \dot \gamma_n(r)\dr r$ associated with the area swept by the path $\gamma_n\in \Gamma_n$ reduces to the expression $\sum_{j=0}^{n-1}z_j\cdot_\sigma \xi_{j+1}$, where $z_j=z_0+\gamma(t_j)$ and $\xi_j=\dot \gamma(s)$ for $s\in (t_j,t_{j+1})$.

According to formula \eqref{general approximations magnetic field}, in the construction of the Feynman  representation for the solution of the Schr\"odinger equation on $\bH^1$ the integration has to be restricted to a peculiar class of paths $\gamma$ on the horizontal space, which contribute to the $s$ variable of the function $\psi(t)$ via their area integral $\int_0^t\gamma(r)\wedge \dot \gamma(r)\dr r$.
In fact, this can be considered to be the analogue of the Wiener integral representation \eqref{Feynman-Kac-magnetic-heat} for the solution of the heat equation.
The renormalization term $ \sum_{k=1}^n \int_0^t e_k (r)\wedge \dot{e}_k (r) \dr r $ due to the implementation of different approximation procedures arises as a consequence of the well known non-continuity of the stochastic integral or, more generally, of the solution of stochastic equations as a function of the driving noise. In fact, the Levy stochastic area is a paradigmatic example of this issue (see e.g. \cite{HeiFri} Proposition1.1), which can be addressed via different techniques of stochastic analysis, ranging from Malliavin calculus to rough path theory \cite{HeiFri,Nua}.

    \section{Construction of the Brownian motion on $ \bH^d$ as weak limit of random walks}

    \label{Sec: Construction of the Brownian motion on Hd as weak limit of random walks}
Below, we review some well-known facts in stochastic processes and refer the reader to \cite{Bil,EthKur} for the proofs.
\subsection{Weak convergence of probability measures on path spaces}
Given a locally compact and separable metric space $(E,d_E)$, we will denote with the symbol $D_E[0,+\infty)$ the set of {\em c\'adl\'ag} $E$-valued paths, i.e. the set of paths $\gamma:[0,+\infty)\to E$ satisfying
\begin{equation*}
 \forall t\geq 0, {\hbox{ exists} \lim_{s\to t^-}\gamma(s)}\equiv \gamma(t^-) \hbox{ and } \lim_{s\to t^+}\gamma(s)= \gamma(t)\,.
\end{equation*}

It is possible to define  a metric $d$ on $D_E[0,+\infty)$ under which it becomes a complete separable metric space. The topology induced by the metric is called {\em Skorohod topology} \cite{Bil,EthKur}. Further, if $(\gamma_n)_n$ is a sequence in $D_E[0,+\infty)$ converging in the Skorohod topology to a continuous path $\gamma \in D_E[0,+\infty)$, then (see \cite{EthKur} Lemma 3.10.1) $(\gamma_n)_n$ converges uniformly to $\gamma$ on compact sets:
\begin{equation}\label{unif-conv-D}
    d(\gamma_n,\gamma)\to 0\,, \gamma\in C_E[0,+\infty) \quad \Rightarrow \quad \forall u>0\, \sup_{t\in [0,u]}d_E(\gamma_n(t),\gamma(t))\to 0
\end{equation}
Let $\cF_D$  denote the  Borel $\sigma-$algebra on $D_{E}[0,+\infty)$. It can be proved that it coincides with the $\sigma$-algebra 
  $\sigma (\pi_t, t\geq 0)$ generated by the projection maps $\pi_t:D_{E}[0,+\infty)\to E$ defined by $\pi_t(\gamma):=\gamma (t)$.
Hence, a stochastic process $X=(\Omega, \cF , \mathbb{P}, (\cF_t)_{t\in\mathbb{R}_+}, (X_t)_{t\in\mathbb{R}_+})$ with  trajectories in  $D_{E}[0,+\infty)$ can be looked at as a $D_{E}[0,+\infty)-$valued random variable, i.e. as a map $X:\Omega \to D_{E}[0,+\infty)$ defined as:
$$X(\omega):=\gamma _\omega, \qquad \gamma_\omega (t):=X(t)(\omega), \qquad t\in [0, +\infty), \, \omega \in \Omega. $$
We shall denote with the symbol $\mu_X$ the distribution of $X$, i.e.  the probability measure  on $\cF_D$ obtained as the pushforward of $\bP$ under $X$:
$$\mu_X(I)=\bP(X(\omega )\in I) \, ,\qquad I\in \cF_D\, .$$

Another important class of sample paths is provided by the set $C_E[0,+\infty)$  of continuous $E$-valued paths, endowed with the topology of uniform convergence on compact sets, i.e $\gamma_n\to \gamma $ in $C_E[0,+\infty)$ if:
\begin{equation}\label{conv-CEinft}
     \sup_{t\in [0,T]}d_E(\gamma_n(t),\gamma(t))\to 0\quad \forall T>0
\end{equation} 
Analogously to the case of stochastic processes with c\'adl\'ag sample paths, given a stochastic process $X\equiv (\Omega, \cF, \bP, (\cF_t)_{t\in\mathbb{R}_+}, (X_t)_{t\in\mathbb{R}_+})$ with continuous trajectories, it can be regarded as a $C_E[0,+\infty)$-valued random variable and it induces a probability measure $\mu_X$ on $\mathcal{B}(C_E[0,+\infty))$.

If $(\mu_n)_n$ and $\mu$ are distributions of random variables $(X_n)_n$ and $X$ taking values in $S$, then, by definition,  $(\mu_n)_n$ weakly convergences to $\mu$ if and only if $(X_n)_n$ converges in distribution to $X$.
This applies in particular if $S=D_E[0,+\infty)$ or $S=C_E[0,+\infty)$, i.e. if $(X_n)_n$ and $X$ are stochastic processes with c\'adl\'ag, respectively continuous, $E$-valued paths.
It is important to point out that whenever $(\mu_n)_n$ weakly converges to $\mu$ then, according to the Skorohod representation theorem, it is possible to construct a probability space $(\Omega, \cF, \bP)$ and random variables $(X_n)_n$ and $X $ on $(\Omega, \cF)$ taking values in $S$ such that $X_n\to X$ almost surely:
\begin{equation*}
    \bP(\lim_{n\to \infty}X_n=X)=1\,.
\end{equation*}

The convergence in distribution of stochastic processes, or more specifically, of weak convergence of probability measures on the most relevant sample path spaces $C_E[0,+\infty)$ and $D_E[0,+\infty)$, is an extensively studied topic \cite{Bil}.  In the case of stochastic processes with continuous sample paths, the following theorem provides an important tool for the proof of their  convergence in distribution. In order to state it, let us recall the definition of the {\em modulus of continuity} $w_T(\gamma)$ of a path $\gamma:[0,T]\to E$, defined as the map $w_T(\gamma):\bR^+\to \bR$
\begin{equation}
    w_T(\gamma)(\delta)\equiv w_T(\gamma,\delta):=\sup_{s,t\in [0,T], |t-s|<\delta}d_E(\gamma(t),\gamma(s))\,.
\end{equation}
The following result is proved in \cite{Bil}, Theorem 7.5.
\begin{theorem}\label{TeoWK-C}
    Let $(X_n)_n$ and $X$ be stochastic processes with sample paths in $C_E[0,+\infty)$. If:
    \begin{enumerate}
        \item for all $k\geq 1$, $t_1,\ldots ,t_k\in [0,+\infty)$ the sequence of random variables $(X_n(t_1),\ldots,X_n(t_k))$ converges in distribution to the random variable $(X(t_1),\ldots,X(t_k))$;
        \item 
        \begin{equation}\label{tightness-1}
    \lim_{\delta\to 0}\limsup_{n} \bP(w_T(X_n,\delta)\geq \epsilon)=0\qquad\forall T>0\,,\,\forall\epsilon>0
\end{equation}
    \end{enumerate}
    then $(X_n)$ converges in distribution to  $X$.
\end{theorem}
 Condition \eqref{tightness-1} is equivalent to the requirement that for all $T>0$, fixed   $\epsilon>0$ and $\eta>0$ it is possible to find $0<\delta<1$ and  $n_0\in \bN$ such that :
\begin{equation*}
         \bP(w_T(X_n,\delta)\geq \epsilon)<\eta, \qquad \forall n\geq n_0\,.
      \end{equation*}

 \subsection{A sequence of random walks associated to the Chernoff approximation}
 First of all, let us point out that the convergence result in \eqref{convergence-cherrnoff-C0} can also be formulated in the following way (see \cite{EN1} Th 5.2 Ch. III)
for all $T\geq 0$ and $f\in C_0(\mathbb{H}^d)$:
\begin{align*}
    \lim_{n\to\infty}\sup_{t\in[0,T]}\Big\|S(1/n)^{\lfloor nt\rfloor }f-V(t)f\Big\|_{C_0(\mathbb{H}^d)}=0\,.
\end{align*}
where $\lfloor \cdot\rfloor$ denotes the integer part.
Let us now consider the sequence $\{X_n(t)\}_{n\geq 1} $ of jump processes on $\mathbb{H}^d$ defined as
\begin{equation}\label{Xnjump}
\begin{cases}
X_n(0)\equiv (z,s) ,\\ X_n(t):=X_n(\lfloor  nt\rfloor/n)=Y_n(\lfloor  nt\rfloor) \quad t>0,
\end{cases}\end{equation}
where $(Y_n)_n$ is a discrete-time Markov process with Markov transition function \footnote{$P: \mathbb{H}^d\times  \mathcal{B}(\mathbb{H}^d)\to   [0,1]$ is the map associated to the discrete time Markov chain $Y=(Y_n)_n$ by $$ \bP(Y_{n+1}\in B|\cF_n^{Y})=P(Y_n,B)\,,$$
where $(\cF_n^{Y})$ is the filtration generated by the process $Y$.} $P((z,s),B)$  given by 
\begin{equation} P((z,s),B)= 
\frac{1}{(2\pi)^d}\int_{\mathbb{R}^{2d}}e^{-|\zeta|^2/2}\chi_B(z+\sqrt{t}\zeta,s+\sqrt{t} z\cdot_\sigma\zeta)\mathrm{d}^{2d}\zeta, \quad B\in \mathcal{B}(\mathbb{H}^d)\,,\end{equation}
with $\chi_B$ being the indicator function of the Borel set $B\in \mathcal{B}(\mathbb{H}^d)$ and $t=1/n$.

By construction, the sample paths of each random walk $(X_n)$ are c\'adl\'ag functions belonging to the space $D_{\mathbb{H}^d}[0,+\infty)$, hence each $X_n$ induces a probability measure $\mu_n$ on the Borel sets of $D_{\mathbb{H}^d}[0,+\infty)$.
By \cite[Theorem 2.6 Ch 4]{EthKur}, (see also \cite[Theorem $19.25$]{Kal}), 
the following holds:
\begin{theorem}\label{thconvD}
    The sequence of processes $X_n$ converges weakly in $D_{\mathbb{H}^d}[0,+\infty)$ and  its weak limit is the Feller process $X$ associated with the Feller semigroup $(V(t))_{t\geq 0}$ \smnote{Actially by the specific choice of the generator $L$, the process $X$ is the Brownian motion on $\mathbb{H}^d$}
\end{theorem}
In other words, according to Theorem \ref{thconvD}, the sequence of probability measures $\{\mu_{X_n}\}_n$ on the Borel $\sigma $-algebra of $D_{\mathbb{H}^d}[0,+\infty)$ converges weakly to the distribution $\mu$ of process $X$, i.e. the Brownian motion on $\mathbb{H}^d$. \\
As a diffusion process, $X$ has continuous sample paths, hence it is meaningful to investigate alternative approximation of $X$ in terms of sequences of random walks on $\mathbb{H}^d$ with trajectories in $C_{\mathbb{H}^d}[0,+\infty)$. 
To this end, let us consider the sequence of processes $(Z_n)_n$ with sample paths obtained by continuous interpolation of the paths of $(X_n)_n$ by means of geodesic arcs.
More specifically, for any $t\in [0,+\infty)$ the random variables $(Z_n(t))_{t\geq 0}$ are defined by
\begin{equation}
    \begin{cases}
    Z_n (0)\equiv (z,s), \\
    Z_n(m/n)\equiv X_n(m/n), \quad m\in \bN,\\
     \quad Z_n(t)=\gamma_{X_n(m/n),X_n((m+1)/n)}(t-m/n), \: t\in [m/n, (m+1)/n]
    \end{cases}
\end{equation}
where $\gamma_{x,y}(t)$ denotes an arbitrary shortest geodesics in $\mathbb{H}^d$ such that $\gamma_{x,y}(0)=x$ and $\gamma_{x,y}(1/n)=y$. 
Let us denote with $\mu_{Z_n}$, resp. $\mu_X$, the Borel measure over  $C_{\mathbb{H}^d}[0,T]$ induced by the process $Z_n$, resp. $X$.  The following holds.
\begin{theorem}\label{teoconvC}
Under the assumptions above, $Z_n$ converges in distribution to $X$ on $C_{\mathbb{H}^d}[0,T]$ (equivalently  the sequence of measures $\{\mu_{Z_n}\}$ converges weakly to $\mu_X$. 
\end{theorem}
The proof of this result relies on Theorem \ref{TeoWK-C}
and on the following lemma (see also \cite{SWW2007}).

\begin{lemma}\label{lemma4}
  let $\nu_n$ be a sequence of probability measures on  $D_E[0,+\infty)$ converging weakly to a finite measure $\nu$ which is concentrated on $C_E[0,+\infty)$. Then for any $T>0$ and for any $\varepsilon >0$ 
  \begin{equation}\label{conv-lemma-DE}
      \lim_{\delta \downarrow 0}\limsup_{n\in \bN}\nu_n(\{w_T(\gamma , \delta )>\varepsilon\})=0
  \end{equation}
\end{lemma}
\begin{proof}[Proof of Lemma \ref{lemma4}]
    By the Skorohod representation theorem there exist a probability space $(\Omega, \cF,\bP)$ and $D_E[0,+\infty)$-valued random variables $(X_n)$, $X$ ( i.e. stochastic processes  with paths in  $D_E[0,+\infty)$) such that $(\nu_n)$ and $\nu$ are the corresponding distributions and $X_n\to X$ a.s. Hence, there exists a set $A\in \cF$, with $\bP(A)=1$, such that for all $\omega \in A$ the sequence of  paths $\gamma_{n,\omega}\equiv X_n(\omega) $ converges to the path $\gamma_\omega\equiv X(\omega)$ in the topology of $D_E[0,+\infty)$. By Eq. \eqref{unif-conv-D}, for any compact set $[0,T]$ $\gamma_{n,\omega}$ converges uniformly on $[0,T]$ to $\gamma_\omega$ for any $\omega \in A$. Since by assumption $\nu$ is concentrated on $C_E[0,+\infty)$, the paths $\gamma_\omega$ are a.s. continuous, hence for all $T>0$ we have $w_T(\gamma , \delta )$ converges a.s. to 0 for $\delta\downarrow 0$. This gives for all $T>0$ the convergence of  $w_T(\gamma , \delta )$  in probability:
    \begin{equation}\label{cont-w-gamma}
   \forall \epsilon >0\qquad      \lim_{\delta \to 0}\bP(w_T(\gamma , \delta )>\epsilon)=0\,.
    \end{equation}
    Let us fix a compact interval $ [0,T]$ and consider now the sample paths $(\gamma_{n,\omega})_n$ of the sequence of processes $(X_n)_n$. For any $\omega\in \Omega$ and $\delta >0$ the following holds:
    $$ w_T(\gamma_{n,\omega} , \delta )\leq w_T(\gamma_\omega, \delta )+2\sup_{t\in [0,T]}d_{\mathbb{H}^d}(\gamma_{n,\omega}(t),\gamma_\omega(t))\,.$$
    Hence, by the a.s convergence of $\gamma_n$ to $\gamma$ and \eqref{cont-w-gamma} we have:
    \begin{equation}
        \label{in-wn-w}
        \bP(w_T(\gamma_{n,\omega} , \delta )>\epsilon)\leq \bP\Big( w_T(\gamma_\omega, \delta )>\epsilon-2\sup_{t\in [0,T]}d_{\mathbb{H}^d}(\gamma_{n,\omega}(t),\gamma_\omega(t)\Big)\,.
    \end{equation}
    Let us now fix $\epsilon >0 $ and $\eta>0$.
    Since $\gamma_n\to \gamma$ a.s., we have that $\sup_{t\in [0,T]}d_{\mathbb{H}^d}(\gamma_{n,\omega}(t),\gamma_\omega(t))\to 0$ in probability, hence there exists an $n_0\in \bN$  such that 
    \begin{equation}\label{con-un-p-gng}
        \bP\Big(\sup_{t\in [0,T]}d_{\mathbb{H}^d}(\gamma_{n,\omega}(t),\gamma_\omega(t))>\epsilon/2\Big)<\frac{\eta}{2},\qquad \forall n\geq n_0\,.
    \end{equation}
    Moreover, by inequality \eqref{in-wn-w} we have:
    \begin{align*}
      &\bP(w_T(\gamma_{n,\omega} , \delta )>\epsilon)
      \leq  \bP\Big( w(\gamma_\omega, \delta )>\epsilon-2\sup_{t\in [0,T]}d_{\mathbb{H}^d}(\gamma_{n,\omega}(t),\gamma_\omega(t))>\epsilon/2\Big)\\
      &+\bP\Big( w_T(\gamma_\omega, \delta )>\epsilon-2\sup_{t\in [0,T]}d_{\mathbb{H}^d}(\gamma_{n,\omega}(t),\gamma_\omega(t)),\sup_{t\in [0,T]}d_{\mathbb{H}^d}(\gamma_{n,\omega}(t),\gamma_\omega(t))\leq\epsilon/2\Big)\\
      &\leq \bP\Big(\sup_{t\in [0,T]}d_{\mathbb{H}^d}(\gamma_{n,\omega}(t),\gamma_\omega(t))>\epsilon/2\Big)
      +\bP\Big(\sup_{t\in [0,T]}d_{\mathbb{H}^d}(\gamma_{n,\omega}(t),\gamma_\omega(t))\leq\epsilon/2, w_T(\gamma,\delta)>\epsilon/2\Big)
      \\
      &\leq \bP\Big(\sup_{t\in [0,T]}d_{\mathbb{H}^d}(\gamma_{n,\omega}(t),\gamma_\omega(t))>\epsilon/2\Big)+\bP(w_T(\gamma,\delta)>\epsilon/2)
    \end{align*}
    By Eq. \eqref{cont-w-gamma}, there exist a $\delta^*$ such that $\bP(w_T(\gamma,\delta)>\epsilon/2)<\eta/2$ for all $\delta <\delta^*$, which together with Eq. \eqref{con-un-p-gng} gives \eqref{conv-lemma-DE}.
\end{proof}

\begin{proof}[Proof of Theorem \ref{teoconvC}]
 Let us consider the trajectories $\gamma_\omega$ of the process $Z_n$, defined as $\gamma_\omega (t):=Z_n(t)(\omega)$. Fix $\delta >0$ and take $n$ in a sufficiently large way that $1/n<\delta$. Consider $s,t\in [0,T]$, $s<t$, $|t-s|<\delta$. We will have   $s\in [m/n,(m+1)/n]$ and $t\in [m'/n,(m'+1)/n]$, with $m\leq m'$, hence:\smnote{in questi passaggi utilizziamo le propriet\'a delle geodetiche}
\begin{align*}
    & d(\gamma_\omega (s), \gamma _\omega (t))\\
    &\leq
    d\left(\gamma_\omega(s),\gamma_\omega((m+1)/n)\right)+d\left(\gamma_\omega((m+1)/n),\gamma_\omega(m'/n)\right)+ d\left(\gamma_\omega(m'/n),\gamma_\omega(t)\right)\\
    &\leq d\left(\gamma_\omega(m/n),\gamma_\omega((m+1)/n)\right)+d\left(\gamma_\omega((m+1)/n),\gamma_\omega(m'/n)\right)+ d\left(\gamma_\omega(m'/n),\gamma_\omega((m'+1)/n)\right)\\
    &\leq 3\max\{ d\left(\gamma_\omega(m/n),\gamma_\omega(m'/n)\right), |m/n-m'/n|<\delta\}
\end{align*}

We can then estimate the probability that the modulus of continuity of the trajectories of $Z_n$ exceeds a given $\varepsilon >0$ as
\begin{eqnarray*}
&&\mu_{Z_n} \left(\{\gamma\in C_{\mathbb{H}^d}[0,T]\colon w(\gamma ,\delta)>\varepsilon\}\right)\\
&&\leq \mu_{Z_n} \left(\{\gamma\in C_{\mathbb{H}^d}[0,T]\colon \max_m\{d(\gamma (m/n),\gamma (m+1)/n))\}>\varepsilon/3\}\right)\\
& &=\mu_{X_n}\left(\{\gamma\in D_{\mathbb{H}^d}[0,T]\colon
w(\gamma ,\delta)>\varepsilon /3\}\right)
\end{eqnarray*}
By Theorem \ref{thconvD} and Lemma \ref{lemma4}, we get for any $\varepsilon>0$
$$\lim_{\delta\downarrow 0}\limsup_n \mu_{Z_n} \left(\{\gamma\in C_{\mathbb{H}^d}[0,T]\colon w(\gamma ,\delta)>\varepsilon\}\right) =0$$
Since $Z_n(0)=(z,s)$ for any $n$, the sequence of probability measures $\{\mu_{Z_n}\}$ is tight \cite{Bil} and the measure $\mu_X$, i.e. the law of $X$, is the only possible limit point.
\end{proof}


\begin{thebibliography}{0}

\bibitem{AlBr}
Albeverio, S., Brze{\'z}niak, Z.:
\newblock Finite-dimensional approximation approach to oscillatory integrals and stationary phase in infinite dimensions.
\newblock J. Funct. Anal. 113(1), 177--244 (1993).

\bibitem{AlBrMF}
Albeverio, S., Brze{\'z}niak, Z.:
\newblock Oscillatory integrals on Hilbert spaces and Schr\"odinger equation with magnetic fields.
\newblock J. Math. Phys. 6(5), 2135--2156 (1995).

\bibitem{AlCaMa}
Albeverio, S., Cangiotti, N., Mazzucchi, S.:
\newblock A rigorous mathematical construction of Feynman path integrals for the Schrödinger equation with magnetic field.
\newblock Communications in Mathematical Physics, 377, 1461-1503. (2020).


\bibitem{AlGuMa}
Albeverio, S., Guatteri, G.,  Mazzucchi, S. (2003). A representation of the Belavkin equation via Feynman path integrals. Probability theory and related fields, 125(3), 365-380.


\bibitem{AlHK}
Albeverio S., H\`oegh-Krohn, R.:
\newblock Oscillatory integrals and the method of stationary phase in infinitely many dimensions, with applications to the classical limit of quantum mechanics.
\newblock Invent. Math. 40(1), 59--106 (1977).

\bibitem{AlHKMa}
Albeverio, S., H\`oegh-Krohn, R., Mazzucchi, S.:
\newblock  Mathematical theory of Feynman path integrals - An Introduction.  2$^{nd}$ corrected and enlarged edition.
\newblock  Lecture Notes in Mathematics, Vol. 523. Springer, Berlin, (2008).


\bibitem{AlMa2}
Albeverio S., Mazzucchi, S.:
\newblock  Feynman path integrals for polynomially growing potentials.
\newblock J. Funct. Anal. 221(1), 83--121 (2005). 

\bibitem{AlMa16}
Albeverio, S., Mazzucchi, S.:
\newblock   A unified approach to infinite-dimensional integration.
\newblock Rev. Math. Phys. 28(2), 1650005--43 (2016).

\bibitem{Bar} Agrachev, A., Barilari, D, Boscain, U. Comprehensive Introduction to sub-Riemannian Geometry. Cambridge University Press, (2020).

\bibitem{App} Applebaum, D., Cohen, S.: Lévy processes, pseudo-differential operators and Dirichlet forms in the Heisenberg group. Ann. Fac. Sci. Toulouse Math. (6) 13, no. 2, 149–177 (2004).

\bibitem{Bau} Baudoin, F., Demni, N., Wang, J.: Stochastic Areas, Horizontal Brownian Motions, and Hypoelliptic Heat Kernels. EMS Tract collection, (2024).

\bibitem{G1} Bahouri, H., Chemin, J.Y., Danchin, R.: A frequency space for the Heisenberg group, Annales de
l’Institut de Fourier, 69, pages 365--407, (2019).
\bibitem{G2} Bahouri, H, Fermanian-Kammerer, C., Gallagher, I.: Dispersive estimates for the Schr\"odinger operator on step $2$ stratified Lie groups, Analysis and PDE, 9, pages 545--574, (2016).

\bibitem{G3} Bahouri, H, Fermanian-Kammerer, C., Gallagher, I.: Phase-space analysis and pseudo-differential calculus on the Heisenberg group, Asterisque, Bulletin de la Societe Mathematique de France, 340, (2012).

\bibitem{G4} Bahouri, H., Gallagher, I.: Paraproduit sur le groupe de Heisenberg et applications, Revista Matematica Iberoamericana, 17(1), pages 69--105, (2001).

\bibitem{Gor} Banerjee, S., Gordina, M., Phanuel, M.:
Coupling in the Heisenberg group and its applications to gradient estimates.
Ann. Probab. 46, no. 6, 3275--3312 (2018).

\bibitem{BLU07} Bonfiglioli, A., Lanconelli, E., Uguzzoni, F.: Stratified Lie groups and Potential Theory for their Sub-Laplacians, Springer Monographs in Mathematics, Springer, Berlin, (2007).

\bibitem{Bil} Billingsley, P.: Convergence of probability measures. Second edition. --- John Wiley \& Sons, Inc., New York, (1999).


\bibitem{Cam}
Cameron, R.H.:
\newblock  A family of integrals serving to connect the Wiener and Feynman integrals.
\newblock J. Math. and Phys. 39(1--4), 126--140 (1960).

\bibitem{Cap}
Capogna, L., Danielli,D., S. D.,Pauls, and Tyson, J. T.: An introduction to the Heisenberg group and the sub-Riemannian isoperimetric problem, volume Progress in Mathematics 259. Birkhauser Verlag, Basel, (2007).

\bibitem{CarGor}
Carfagnini, M., Gordina, M.: 
\newblock  Dirichlet sub-Laplacians on homogeneous Carnot groups: spectral properties, asymptotics, and heat content.
\newblock IMNR, 2024:3, 1894--1930 (2024).
\bibitem{Cher}
Chernoﬀ, P.R.:
\newblock  Note on product formulas for operator semigroups.
\newblock J. Functional Analysis 2:2, 238--242 (1968).

\bibitem{Dos}
Doss, H.:
\newblock Sur une r\'esolution stochastique de l'\'equation de Schr\"odinger \'a  coefficients analytiques. 
\newblock Comm. Math. Phys. 73(3), 247--264 (1980).


\bibitem{DrMaMo23}
Drago, N., Mazzucchi, S., Moretti, V.: 
Feynman path integrals on compact Lie groups with bi-invariant Riemannian metrics and Schrödinger equations.
arXiv:2307.03282 




\bibitem{Dui}
Duistermaat, J.J.:
\newblock Oscillatory integrals, Lagrange inversions and unfolding of singularities.
\newblock Comm. Pure Appl. Math. 27(2), 207--281 (1984).

\bibitem{ELT}
Elworthy, D., Truman, A.:
\newblock Feynman maps, Cameron--Martin formulae and anharmonic oscillators.
\newblock Ann. Inst. H. Poincar\'e Phys. Th\'eor. 41(2), 115--142 (1984).


\bibitem{EN1} Engel, K.J., Nagel, R.: One-Parameter Semigroups for Linear Evolution Equations. --- Springer, (2000).

\bibitem{EthKur} Ethier, S. and   Kurtz, T.:
\newblock Markov processes. Characterization and convergence. --- John Wiley \& Sons, Inc., New York, (1986). 

\bibitem{FeyThesis}
Feynman, R.P., and  Brown L. M.:
\newblock Feynman's thesis: a new approach to quantum theory. World Scientific, (2005).

\bibitem{Fey}
Feynman, R.:
\newblock Space-time approach to non-relativistic quantum mechanics.
\newblock Rev. Mod. Phys. 20, 367-387 (1948).

\bibitem{Fis}
Fischer V., Ruzhansky M.: Quantization on nilpotent Lie groups, Progress in Mathematics, Vol. 314, Birkhauser, (2016). 

\bibitem{HeiFri}
Friz, P., Hairer, M.:
\newblock A Course on Rough Paths. Springer (2020)

\bibitem{Fu}
Fujiwara, D.:
\newblock  Rigorous Time Slicing Approach to Feynman Path Integrals. Springer Japan (2017).


\bibitem{Fuk}
Fukushima S.,
Time-slicing approximation of Feynman path integrals on compact manifolds,
Annales Henri Poincaré. Vol. 22. No. 11. Cham: Springer International Publishing, (2021).

\bibitem{HKPS} Hida, T., Hui-Hsiung Kuo, Potthoff, J., Streit, W.:
\newblock White Noise. An Infinite Dimensional Calculus.
\newblock Kluwer, Dordrecht (1995).

\bibitem{Hor}
H\"ormander, L.:
\newblock  The analysis of linear partial differential operators. I. Distribution theory and Fourier analysis. Reprint of the second (1990) edition.
\newblock  Classics in Mathematics. Springer-Verlag Berlin (2003).  

\bibitem{Ichi0}
Ichinose, W.:
\newblock On the Formulation of the Feynman Path Integral Through Broken Line Paths.
\newblock Comm. Math. Phys. 189(3), 17--33 (1997). 


\bibitem{JoLa} 
Johnson, G.W., Lapidus, M.L.:
The Feynman integral and Feynman's operational calculus.
\newblock Oxford University Press, New York (2000).



 \bibitem{Kal}
Kallenberg, O.: Foundations of Modern Probability (2nd ed.). --- Springer-Verlag, (2002).

\bibitem{KarSh}
Karatzas, I., Shreve, S.E.:
\newblock Brownian motion and stochastic calculus.
\newblock Springer-Verlag, New York (1991).


\bibitem{Mabook}
Mazzucchi, S.:
\newblock  Mathematical Feynman Path Integrals and Applications. Second edition.
\newblock  World Scientific Publishing, Singapore (2022)

 
\bibitem{MaNiTr}
  Mazzucchi, S.; Nicola, F.;  Trapasso, S.V..
 \newblock Phase space analysis of finite and infinite dimensional Fresnel integrals.
\newblock arXiv:2403.20082. To appear in {\em J. Funct. Anal.} (2025).


\bibitem{Mur}
Murray, J.D.:
\newblock Asymptotic analysis. 
\newblock Clarendon Press, Oxford (1974). 


\bibitem{Nel}
Nelson, E.:
\newblock  Feynman integrals and the Schr\"odinger equation. 
\newblock  { J. Math. Phys.} 5(3), 332--343 (1964).


\bibitem{Neu}
Neuenschwander, D.:
\newblock Probabilities on the Heisenberg group: limit theorems and Brownian motion.
\newblock Springer (1996).

\bibitem{Nic16}
Nicola, F.:
\newblock Convergence in $L^p$ for Feynman path integrals.
\newblock Adv. Math. 294, 384--409 (2016).

\bibitem{NicTra}
Nicola, F., Trapasso, S.I.:
\newblock On the pointwise convergence of the integral kernels in the Feynman-Trotter formula.
\newblock Communications in Mathematical Physics 376.3, 2277-2299 (2020).

\bibitem{NiTrabook}
Nicola, F., Trapasso, S.I.: Wave Packet Analysis of Feynman Path Integrals. Berlin/Heidelberg, Germany: Springer, (2022).

\bibitem{Nua}
Nualart, D.:
\newblock The malliavin Calculus and related Topics. Springer (2006).

\bibitem{ReSi}
Reed, Michael, Simon, B.: II: Fourier analysis, self-adjointness. Vol. 2. Elsevier, (1975).

\bibitem{Simon}
Simon, Barry. Functional integration and quantum physics. No. 351. American Mathematical Soc., (2005).


\bibitem{SWW2007}  Smolyanov, O.G., Weizs\"{a}cker, V.H., Wittich, O.: Chernoff's Theorem and Discrete Time Approximations of Brownian Motion on Manifolds. Potential Analysis, 26: 1, 1--29 (2007).

\bibitem{Str86} Strichartz, R.: Sub-Riemannian geometry. Journal of Differential Geometry, 24(2), 221--263, (1986).


\bibitem{Tha1}
Thaler H.,
Solutions of Schr\"odinger equations on compact Lie groups via probabilistic methods, Potential Analysis 18, 119--140 (2005).

\bibitem{Tha2}
Thaler H.,
The Doss trick on symmetric spaces,
Letters in Mathematical Physics 72, 115-127 (2005).

\bibitem{Tho}
Thomas, E.:
\newblock Projective limits of complex measures and martingale convergence.
\newblock Probab. Theory Related Fields 119(4), 579--588 (2001).


\bibitem{Tru1}
Truman, A.:
\newblock Feynman path integrals and quantum mechanics as $\hbar \to 0$
\newblock J. Math. Phys. 17(10), 1852--1862 (1976).


\bibitem{Yam} 
Yamasaki, Yasuo. Measures on infinite dimensional spaces. Vol. 5. World Scientific, (1985).
\end{thebibliography}
\end{document}